\documentclass{article}

\usepackage[margin=3.5cm]{geometry}

\usepackage{graphicx}
\usepackage{subcaption}
\usepackage{amssymb}
\usepackage{amsthm}
\usepackage{dsfont}
\usepackage{mathrsfs}
\usepackage[lite,initials,msc-links, alphabetic,nobysame]{amsrefs}
\usepackage{amsmath}
\usepackage{centernot}
\usepackage{bm}
\usepackage{mathtools}
\usepackage{tikz-cd} 
\usepackage{enumerate}
\usepackage{comment}
\usepackage{float}
\usepackage[labelfont=bf]{caption}
\usepackage{titlesec}
\titleformat{\subsection}[runin]{\normalfont\bfseries}{\textmd{\thesubsection.}}{3pt}{}
\titleformat{\part}[block]{\normalfont\bfseries\Large}{\partname\nobreakspace\thepart}{20pt}{\normalfont\bfseries}
\usepackage{centernot}
\usepackage{xcolor}
\usepackage[title,titletoc]{appendix}
\usepackage{hyperref}
\newtheorem{theorem}{Theorem}
\newtheorem{lemma}[theorem]{Lemma}
\newtheorem{proposition}[theorem]{Proposition}
\newtheorem{corollary}[theorem]{Corollary}
\newtheorem{asum}[]{Assumption}
\numberwithin{equation}{section}
\theoremstyle{definition}

\usepackage{chngcntr}
\counterwithout{equation}{section}

\theoremstyle{definition}
\newtheorem{definition}{Definition}
\theoremstyle{remark}
\newtheorem{remark}{Remark} 
\allowdisplaybreaks

\newcommand{\con}[1]{ \xleftrightarrow{#1}}

\DeclareMathOperator{\sgn}{\textnormal{sgn}}

\renewcommand{\d}{\mathrm{d}}

\newcommand{\E}{\mathbb{E}}
\newcommand{\id}{\mathds{1}}

\newcommand{\RR}{\mathbb{R}}

\newcommand\vp{\varphi}
\newcommand\Ec{\mathcal{E}}
\newcommand\R{\mathbb{R}}
\renewcommand{\d}{\mathrm{d}}
\newcommand{\Z}{\mathbb{Z}}
\newcommand{\w}{\omega}
\renewcommand{\l}{\lambda}
\newcommand{\vpp}{\vp^+}
\newcommand{\vpm}{\vp^-}
\newcommand{\I}{\Phi}
\renewcommand{\P}{\mathrm{\mathbf{P}}}
\newcommand{\Q}{\mathrm{\mathbf{P}}}
\renewcommand{\E}{\mathrm{\mathbf{E}}}
\renewcommand{\L}{\xi}
\newcommand{\Lp}{\xi^+}
\newcommand{\Lm}{\xi^-}
\renewcommand{\i}{\mathds{1}}
\renewcommand{\S}{\mathcal{S}}
\newcommand{\vps}{ \widetilde{\vp}}
\newcommand{\eps}{\varepsilon}
\newcommand{\G}{\Gamma}
\newcommand{\g}{\mathfrak{g}}
\renewcommand{\H}{\mathbf{H}}
\newcommand{\F}{\mathcal{F}}
\newcommand{\Vis}{V_{\mathrm{Ising}}}
\renewcommand{\Vis}{V_{\pm 1}}
\newcommand{\Sis}{S_{\pm 1}}
\newcommand{\Pf}{\mathrm{{Pf}}}
\renewcommand{\k}[1]{{#1}^{(k)}}
\renewcommand{\path}{\tau}

\hypersetup{colorlinks=true,linkcolor=blue,citecolor=blue}

\makeatletter
\newcommand\connect[2][]{%
  \ext@arrow 9999{\longleftrightarrowfill@}{#1}{#2}}
\newcommand\longleftrightarrowfill@{%
  \arrowfill@\leftarrow\relbar\rightarrow}
\makeatother

\title{{Double cluster swapping for spin models: \\Pfaffian relations and sharpness}}

\date{\today}

\author{\ Diederik van Engelenburg$^*$ \and Lorca Heeney$^*$ \and  Marcin Lis\thanks{ Technische Universität Wien.\\ Emails: \texttt{\{diederik.engelenburg, lorca.heeney-brockett, marcin.lis\}@tuwien.ac.at}
}}

\begin{document}
%
\maketitle
\begin{abstract}
We introduce and study a new geometric representation for general classical spin models. 
It consists of two coupled percolation configurations that are the joint FK (random cluster) representation of the Ginibre rotation of two independent copies of the spin model, and can be viewed as an extension of Sheffield's cluster swapping defined in the context of height functions~\cite{She}.
Our approach combines the advantages of the random current and FK representations of the Ising model: it provides a percolation interpretation for various truncated correlation functions and at the same
time satisfies the FKG inequality (for a large subclass of models).

We highlight its strength and versatility by establishing two very different results. We first prove the converse of the classical fact that boundary multi-point correlation functions of planar Ising models
are given by Pfaffians of the two-point functions. Indeed, we show that if a general spin model on a general graph satisfies the Pfaffian relations, then up to natural local modifications it must actually be an Ising model on a planar graph. In particular, the algebraic Pfaffian relations imply the topological feature of planarity. 
Our second application is a proof of sharpness of the phase transition for a class of spin models first considered by Ellis, Monroe and Newman~\cite{ellis1976ghs}, which we do by generalising the celebrated argument of Duminil-Copin and Tassion \cite{DCT}, replacing the use of the random current with our new representation.
\end{abstract}
\tableofcontents



\section{Introduction}






\paragraph{Classical spin models.}
In this article, we will consider real-valued classical spin models governed by the Gibbs measure
\[
	\left< F(\vp) \right>_{}^{} \propto \int_{\R^V} F(\vp) \exp\Big( \beta \sum_{xy \in E} \vp_x \vp_y \Big) \,\prod_{x \in V} \d\l(\vp_x)
\] 
where $\vp \in \R^V$ is a spin configuration on a graph $G = (V,E)$, $\beta > 0$ is the inverse temperature and $\l$ is a symmetric measure on $\R$ which will often arise in the form
\[
\d\l(\vp) \propto \exp(-P(\vp)) \d\vp
\] 
for $P$ an even function of super-quadratic growth which we refer to as the \emph{potential}. For now, we will take the existence of such measures on infinite graphs for granted.

\paragraph{Phase transitions. }In the setting of the hypercubic lattice $\Z^d$ ($d \geq 1$), these models have a long history in the mathematical physics community as lattice approximations to the Euclidean $P(\vp)$-models \cite{GRS}. The Ising model, where
\begin{align} \label{eq:isingmeas}
	\l = \tfrac{1}{2}\delta_1 + \tfrac{1}{2}\delta_{-1},
\end{align}
and the $\vp^4$ model (where $P$ is a quartic even polynomial) are certainly the most famous and well understood within this class, and are expected to belong to the same universality class. Classical spin models are in general known to undergo a phase transition when $d \geq 2$, in the sense that the critical point defined e.g.~in terms of the decay of the two-point function:
\[
	\beta_c := \inf \left\{ \beta \geq 0 \mid \inf_{x \to \infty} \left<\vp_o \vp_x \right>_{\Z^d,\beta} > 0 \right\} 
\] 
is non-trivial \cite{Pei36,glimm1975phase}. Here $\langle \cdot \rangle_{\Z^d,\beta} $ is the infinite volume measure on $\mathbb Z^d$ at inverse temperature $\beta$. Arguably the most interesting questions concern the behaviour of the model exactly at or around the critical point. The transition is said to be \emph{continuous} if the correlation function continues to decay to $0$ at $\beta_c$, and \emph{discontinuous} or \emph{first-order} if this is not the case. The order of the phase transition will vary depending on the potential $P$~\cite{FriVel}. 
What is expected to be universal is that for every potential the transition is \emph{sharp}. That is, the correlation functions decay exponentially fast below the critical point. 

\paragraph{Rigorous results. }It is now known that the phase transition of the Ising model is continuous \cites{Yan52,AF,ADCS} and this has also recently been extended to the $\vp^4$ model when $d \geq 3$ \cites{GPPS}. This extension used the fact that the $\vp^4$ potential belongs to the special \emph{Griffiths-Simon (GS) class} \cite{simon1973varphi4}, which loosely means that it can be directly approximated by Ising models on a modified graph. Using the GS class may look like an appealing strategy to extending these results further to other potentials, however, it is rather opaque and directly determining if a potential lies in this class is an extremely hard problem{\hypersetup{linkcolor=black}\footnote{In fact, a complete classification may have implications for the Riemann Hypothesis \cite[Section 4]{newman}.}}. It is known for example that if $\lambda P + \alpha \vp^2$ belongs to the GS class for every $\lambda > 0$ and $\alpha \in \R$, then $P$ must be a quartic polynomial \cite{newman1976rigorous}. As such, a different approach is needed when considering more general models.

\paragraph{Random currents. }The main reason why progress of the Ising model has surged in recent times is the use of the \emph{random current representation}, introduced in \cites{GHS,Aiz82}. Along with a piece of combinatorial magic called the \emph{switching lemma} (see \cite{duminil2018random} for an introduction), it has allowed for a fine study of phenomena in the Ising model by translating them into questions of percolation and geometry. A non-exhaustive list of results of this kind are the aforementioned continuity and sharpness, triviality in dimension $d \geq 4$, computation of mean-field exponents and emergent planarity in two-dimensional Ising models \cites{AF,ADCS,ABF,DCT,Aiz82,ADC,DCP-2-point,vEGPS,ADTW}. The representation has also been extended to the $\vp^4$ model, again relying on its place in the GS class, and enabling progress there \cite{GPPS}. Constructing a suitable random current for potentials outside the GS class is, as far as we know, out of reach. The main advantage of random currents in short is that they give a direct probabilistic interpretation of differences in and products of correlation functions. 

\paragraph{FK representations. }Another tool in the study of the Ising model is the FK-Ising representation \cite{FK}. More blunt than the random current, it does have certain advantages, foremost the FKG inequality \cite{FKG} and a direct coupling with the spin model. 
A common procedure for generalising it to other real-valued spins systems $\vp$ is to find some suitable conditioning (often on $|\vp|$) so that the conditional law of $\sgn(\vp)$ is a ferromagnetic Ising model (with coupling constants depending on $|\vp|$) and then use the conditional FK-Ising representation. This has be applied in many situations, extending FK-representations to any reference measure $\lambda$ but also outside this class e.g.~to the Ashkin-Teller model and the loop $O(n)$ model to name just two \cite{PfiVel,GlaPel,Lis22a,glazman2025planar}. Moreover, various FK-representations are now used in the general setting of random surface models \cite{Lammers,lammers2022dichotomy,Lis20,Lis22b,glazman2025delocalisation} after they were pioneered by Sheffield in his PhD thesis~\cite{She}.  

\paragraph{A new representation.}One contribution of this work is to introduce a new representation for any classical spin model. 
The construction starts by rotating (up to a factor of $\sqrt 2$) two independent copies $\vp$ and $\vp'$ of the system by $45^\circ$:
\[
	(\vp,\vp') \longrightarrow (\vp^+,\vp^-) \quad \text{ where } \quad \vp^\pm = \left( \vp \pm \vp' \right) / 2.
\]
This is sometimes called the Ginibre rotation as it was famously first used in the proof of the Ginibre inequality~\cite{Gin}.
Despite the pair $(\vp^+,\vp^-)$ no longer having independent entries, we can still construct a joint FK-representation $(\w^+,\w^-)$ for $ (\sgn(\vp^+),\sgn(\vp^-))$. One half of this representation (just $\w^-$) appears under the name \emph{cluster swapping} in the landmark work mentioned above~\cite{She}, since flipping the sign of $\vp^-$ on any cluster of $\w^-$ is equivalent to the swap $\vp \leftrightarrow \vp'$. However, it is properties of the pair $(\w^+,\w^-)$ that will play the main role in our considerations. 
In light of this we chose to call the new representation \emph{double cluster swapping}.

It turns out that it combines some of the main advantages of both double random currents and the FK representation:
\begin{itemize}
\item One can write various truncated correlation functions, such as 
\[
\left< \vp_x \vp_y \right> - \left<\vp_x  \right> \left<\vp_y \right> \quad \text{ and } \quad \left<\vp_x \vp_y \vp_z \vp_w \right> - \left<\vp_x \vp_y \right> \left<\vp_z \vp_w \right>,
\] 
	in terms of certain connection events for $\w^+$ and $\w^-$.
\item The configurations $\w^+$ and $\w^-$ are naturally coupled with the spin configurations $\vp^+$ and $\vp^-$ (and hence also with $\vp$ and~$\vp'$).
\item The quadruple $(|\vp^+|,-|\vp^-|,\w^+,-\w^-)$ satisfies the \emph{FKG inequality} for a natural subclass of potentials originally considered in \cite{ellis1976ghs}. 
In particular, for such measures, $\w^+$ and $\w^-$ are negatively correlated with each other.
\end{itemize}

\paragraph{Two applications. } We use double cluster swapping to prove two very different results. 
\begin{itemize}
	\item We show that the planar Ising model is unique, in a matter that we will make precise, in satisfying the \emph{Pfaffian relations}. The Pfaffian relations are a system of equations concerning the correlation functions of a spin model and are the fermionic counterpart to Wick's rule for Gaussian systems. They were proved to hold for the boundary spins in planar Ising models in \cite{GBK} and one of our main theorems may be seen as the converse to this celebrated result. Remarkably, the algebraic Pfaffian relations imply the topological feature of planarity.
		To the best of our knowledge no statement of this type is achievable through the known representations of spin models. 

	\item We prove that the phase transition for any potential $P$ belonging to the explicit class defined in \cite{ellis1976ghs} is sharp. We do this by generalising the argument of Duminil-Copin and Tassion \cite{DCT}, replacing the use of the double random current with the new representation. This result was new for any potential not belonging to the GS class until a recent work by Panagiotis and Veitch \cite{panagiotis2026subcritical} that appeared during the preparation of this manuscript (see Remark \ref{remark:panagiotis}). The proof in~\cite{panagiotis2026subcritical} applies to any potential $P$ and the methods are disjoint from ours (see Remark \ref{remark:panagiotis}). We believe that our approach has value nonetheless since it is comparatively elementary and also applies in the infinite-range setting.
\end{itemize}

\section{Main results}
We always take $\l$ to be a symmetric probability measure on $\R$, which we call the reference measure. An important class of examples is when $P : \R \to \left[ 0,\infty \right]$ is an even function such that the measure $\l = \l^{(P)}$ given by 
\[
\d\l(\vp) \propto \exp( - P(\vp)) \d \vp,
\]
is well-defined. The reference measure $\l$ is always assumed to be non-trivial ($\l \neq \delta_0$). We (almost) always assume the following super-Gaussian tail assumption on $\l$:
\begin{asum}\label{asum:one}	The reference measure satisfies		
\[
				\int_\R \exp( \alpha \vp^2) \d\l(\vp) < \infty \quad \text{ for every }\alpha \geq 0.
\]
\end{asum}
The one exception to this is that the results of Part \ref{part:1} also apply to the Gaussian free field itself (i.e.~when $P$ is a quadratic polynomial). Under this assumption we can precisely define the spin models for which our results hold.
\begin{definition}[]\label{def:spinmodel}
Let $G=(V,E)$ be a finite graph and let $\vp \in \R^V$ be distributed according to
\begin{equation}\label{eq:spinsystem}
	\left< \cdot \right>_{G,\beta,h} = \left< \cdot  \right>_{G,\beta, h, \l} \propto \exp\Big( \beta \sum_{uv \in E} J_{uv} \vp_u \vp_v  + h \sum_{u \in V}  \vp_u \Big) \prod_{u \in V} \d\l(\vp_u)
\end{equation} 
where $J : E \to [0,\infty)$ are the coupling constants, $\beta \geq 0$ is the inverse temperature, $h \in \R$ is the external magnetic field and $\l$ is a reference measure satisfying the assumptions above. 
\end{definition}
Any of $G,\beta,h$ and $\l$ may be dropped from the notation $\left< \cdot \right>_{G,\beta,h,\l}$ where clear.
More generally, we may even allow for a family of reference measures $(\l_v)_v$ which vary across the vertices of $v \in V$.
\\[1em]
\noindent These spin models have been extensively studied and we recommend \cite{simon2026phase} for a comprehensive introduction.
\subsection{Application I: Pfaffian relations.}\label{sec:app2}\mbox{}\par
\paragraph{Background. }Pfaffian relations are a set of equations determining how higher-order correlations may be expressed in terms of the second-order (or two-point) correlations. They are the fermionic counterpart to Wick's rule for the bosonic free field. It is well known since the work of Groeneveld, Boel and Kasteleyn~\cite{GBK} 
that spin correlations of an Ising model on a planar graph
satisfy Pfaffian relations (see Definition~\ref{def:pfaf} below) 
whenever the spins are placed on the boundary of the graph (or more generally on a single face of the graph). 
In recent years new graphical proofs of this identity appeared using 
the double random current representation~\cites{LisT,ADTW}. Another possible approach is to directly use the exact solvability of 
the planar Ising model in the form of its dimer representation on the Fisher graph~\cite{FisherDim} 
or the Kac--Ward solution~\cite{KacWard}. We also mention that Pfaffian relations have been shown to hold asymptotically at the critical point of finite-range Ising models on $\Z^2$ \cite{ADTW} and it is conjectured that they also hold asymptotically for the critical $\vp^4$ model in the same setting \cite[Problem 2]{GPPS}.

\subsection{Setup and result.} Let us fix $G=(V,E)$, $J$ and $\beta > 0$ as above, and also a family $(\l_v)_{v\in V}$ of reference measures. We will use the shorthand
\[
	\left<  \cdot \right>_{G} = \left< \cdot \right>_{G,\beta,J,0,(\l_v)_v}
\]
for the corresponding spin model defined in Definition \ref{def:spinmodel} with zero magnetic field.
\begin{definition} \label{def:pfaf}
	Let $W =\{ v_1,v_2,\ldots, v_{n} \} \subseteq V$.
	We say that $\left<\cdot \right>_G$ satisfies \emph{Pfaffian relations} on the set $W$ if for any subset $I\subseteq [n]$,
	\begin{equation}\label{eq:Pfeq}
        \Big \langle  \prod_{i\in I}\vp_{v_i} \Big \rangle_G = \Pf(M_I),
    \end{equation}
    where $M$ is the $n\times n$ skew-symmetric matrix with 
    $M_{i, j} = \langle \vp_{v_i} \vp_{v_j} \rangle$ for $i<j$, 
    and $M_I$ is its restriction to $I \times I$. 
    The Pfaffian of $M_I$, denoted $\Pf(M_I)$ here, is the positive square root of the determinant of $M_I$.
\end{definition}
\noindent We note that this definition depends on a (cyclic) ordering of the vertices in $W$ which we leave implicit in the notation, and is only non-trivial when $n$ is even and at least four. 	
As an example, if $W = \left\{ v_1,v_2,v_3,v_4 \right\}$ is of size $4$ then the Pfaffian relation on $W$ is exactly
\begin{equation}\label{eq:pfaff4}
	\left<\vp_{v_1} \vp_{v_2} \vp_{v_3} \vp_{v_4} \right>_G = \left<\vp_{v_1} \vp_{v_2} \right>_G \left< \vp_{v_3} \vp_{v_4} \right>_G + \left< \vp_{v_1} \vp_{v_4} \right>_G \left< \vp_{v_2} \vp_{v_3} \right>_G - \left< \vp_{v_1} \vp_{v_3} \right>_G \left<\vp_{v_2} \vp_{v_4} \right>_G.
\end{equation} 
	This is graphically depicted in~Fig~\ref{fig:fourpt}, showing that the sign of each term is determined by the parity of its diagram's number of intersections.

\begin{figure}
	\centering
	\includegraphics[scale=0.6]{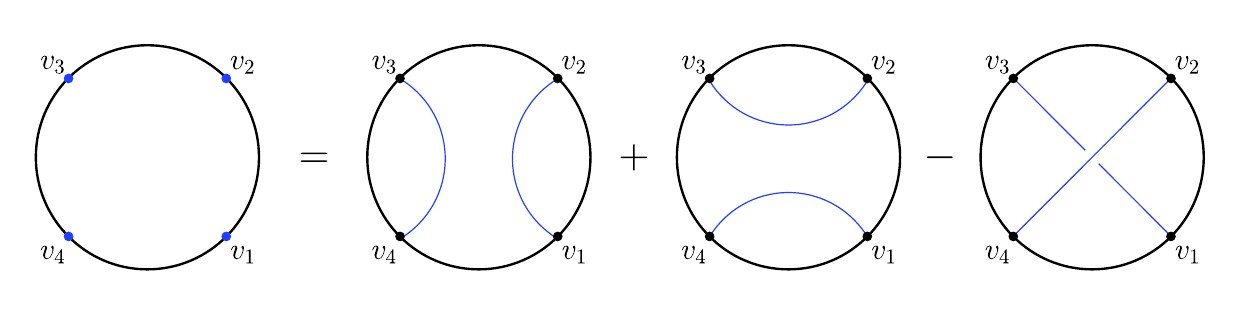}
	\caption{A diagramatic representation of the Pfaffian relation for four points.}
	\label{fig:fourpt}
\end{figure}

To state the result we need two further definitions. Firstly, suppose $V' \subset V$ is a  subset of the vertices of $G$. We define the \emph{path-induced} graph $G'=(V',E')$ to have vertex set $V'$ and edge set $E'$, where $uv \in E'$ if and only if there exists a simple path in $G$ connecting $u$ and $v$ and avoiding $V'$ except at its endpoints. 
The path-induced graph can alternatively be constructed by a sequence of three types of local moves (see Section \ref{sec:graphred}). 
Our second definition is that we say that a vertex $v$ is of \emph{Ising type} if the reference measure $\l_v $ satisfies
\[
	\l_v \propto \delta_s + \delta_{-s}
\] 
for some $s > 0$. 
Note that up to a change of the coupling constants and a vertex-wise rescaling of correlation functions, this is equivalent to \eqref{eq:isingmeas}. 
An instance of Definition \ref{def:spinmodel} is an Ising model on $G$ if each vertex is of Ising type.

We may now present our first main result.
\begin{theorem} \label{thm:pfaffian}
	Let $\langle \cdot \rangle_G$ be a classical spin system on a finite graph $G=(V,E)$ 
    satisfying Pfaffian relations on $W\subseteq V$. 
    Then, 
	\begin{itemize}
		\item 
    there exists a set $U \subseteq V$ of vertices of Ising type, 
    such that the path-induced graph $G'$ from $G$ by $U\cup W$ is a planar graph with $W$ lying on one of its faces.
    		\item 
			Moreover, there exists an Ising model $\sigma \sim \left< \cdot  \right>_{G'}$ on $G'$ with ferromagnetic (nonnegative) coupling constants $J'$ such that 
    for all $I \subset W$,
    \[
        {\Big \langle \prod_{v \in I} {\vp_{v} } \Big\rangle_G}
	= \Big \langle \prod_{v\in I} \sigma_{v}\Big \rangle_{G'} .
    \]
	\end{itemize}
\end{theorem}
That is, we show that if the correlation functions of a classical spin model on a finite graph $G = (V, E)$
satisfy Pfaffian relations on a subset of $n$ vertices, then they match those of a ferromagnetic Ising model on a planar graph with $n$ boundary vertices.
Moreover, we explicitly construct such an Ising model by showing that there must have \emph{already existed} a subset of vertices of Ising type (see below) which together with $W$ induced a planar graph (whose connectivities come naturally from $G$). 
Such a result is, to the best of our knowledge, unattainable using the known methods where precise and special combinatorial properties of the Ising model are used from the start.

\begin{remark}
 Interestingly, it is not true that the vertices in $W$ must be of Ising type themselves for the Pfaffian relations to hold. See Remark \ref{remark:expexamp} for an explicit example.
\end{remark}

An interesting corollary of (the proof of) Theorem \ref{thm:pfaffian} is the following.
\begin{corollary}\label{cor:thm1}
	Let $\left< \cdot \right>_{G}$ be a classical spin model on $G = (V,E)$ and let $W \subset V$. Suppose that for each subset $W' \subset W$ of size $4$ the Pfaffian relation on $W'$ (i.e.~equation \eqref{eq:pfaff4}) is satisfied. Then, $\left< \cdot \right>_{G}$ satisfies the Pfaffian relations on the entire set $W$.
\end{corollary}
\subsection{Proof strategy. }The first step is to prove a novel geometric characterisation of the Pfaffian relations (Lemma \ref{lemma:keylemma}). The geometric condition concerns the probability of certain connection events for the pair $(\w^+,\w^-)$ vanishing. This already has many consequences:
\begin{itemize}
	\item It leads to a new proof of Pfaffian relations for the planar Ising model using the ideas from \cites{Lis22a}. This is based on the key observation that $\w^+$ and $\w^-$ may only meet at a vertex which is not of Ising type.
	\item It immediately proves Corollary \ref{cor:pfaff}.
	\item Most importantly, it forces strong constraints on the geometry of $G$ whenever $\left< \cdot \right>_{G}$ is a spin model on $G$ satisfying the relations. 
Roughly, if the Pfaffian relations are to hold then the Ising type vertices must be situated in order to `block' certain paths of $\w^+$ and $\w^-$ from crossing over one another. 
\end{itemize}
This geometric criterion could also be proved using the random current if we assumed to begin with that $\left< \cdot \right>_{G}$ is an Ising model, the real strength of using the new representation is that we \emph{do not assume it}. 

The rest of the argument is an intricate combination of both graph theory and properties of classical spin models. We develop an algorithm which applies certain reduction moves to the graph
\[
	G \to G^{(1)} \to G^{(2)} \to \ldots
\]
At each step, we crucially prove that the spin model $\left< \cdot \right>_{G}$ itself can be transformed onto $G^{(k)}$ whilst preserving the boundary correlation functions. This is significantly complicated by the fact that the vertices of $W$ themselves are not of Ising type (which can happen, see Remark \ref{remark:expexamp}). 
We must use the constraints on the geometry of the graph to prove that it is possible to remove all non-Ising vertices, so that the algorithm produces in finite time a fully reduced graph $G'$ and an Ising model on $G'$ with the same correlation functions as we started with.  
Finally, we show the final graph $G'$ to be planar using a purely combinatorial result (Proposition \ref{prop:main}) originally due to \cite{RobSey}. A different proof of this is included in Appendix \ref{sec:appendix}. 
\subsection{Application II: Sharpness.}\label{sec:app1}\mbox{}\par

\paragraph{Background. }(Subcritical) sharpness may be considered the statement that there is no \emph{intermediate phase}: the system transitions immediately from a strongly disordered phase to a phase with long-range order. 
It was first proved for the Ising model (and then extended to the $\vp^4$ and GS class) in \cite{ABF}. An alternative simpler argument was given in \cite{DCT}, which we briefly describe below. In \cite{duminil2019sharp}, a new robust approach was developed using a tool called the OSSS inequality. This has lead to a host of sharpness results including for the Blume-Capel and Ashkin-Teller models \cite{gunaratnam2024existence,aoun2024phase}. Still, Theorem~\ref{thm:sharpness} covers many cases for which the result was not previously known (see Remark~\ref{remark:panagiotis}). 
We stress that the supercritical counterpart of sharpness, i.e.~exponential decay of truncated correlations at $\beta > \beta_c$, is considered significantly more difficult and is only known for the Ising model \cite{duminil2020exponential}, although there has been some recent progress for other spin models \cite{GPPS2,gunaratnam2026supercriticalsharpnessrandomcluster}. We hope that the representation developed here could also be useful in attacking this problem, especially in light of \eqref{eq:corr2} and \eqref{eq:isingtrunc}. 

\begin{definition}[]
	We consider the class $\Ec$ of potentials $P$ which can either be written as 
	\begin{equation}\label{eq:Pdef}
	P(x) = \begin{cases}
		C \int_0^{|x|}{ g(t) \d t} \quad & |x| \leq A\\
		\infty \quad & |x| > A
	\end{cases} 
	\end{equation} 
	for some $A \in (0,\infty]$, $C > 0$ and a convex function $g : [0,A] \to \R$ with $g(0) = 0$, or are the weak limit of such measures. 
\end{definition}
\noindent 
If $P$ is three-times continuously differentiable then the condition $P \in \Ec$ is equivalent to
\[
	P'''(x) \geq 0 \quad \text{ for all $x \geq 0$.}
\]
The geometric interpretation of this requirement should be that the potential $P$ is at most \emph{double-welled} (see Figure \ref{fig:doublewell}). 
Examples of potentials in this class include: 
\begin{itemize}
	\item The Ising model (in fact, this is the only measure in $\Ec$ not as in \eqref{eq:Pdef}) 
	\item The $\vp^4$ model, and more generally any even polynomial $P$ with non-negative coefficients of degree four and higher.
	\item The square well potential $\exp(-P(x))\d x \propto \i_{\left[ -1,1 \right]}$.
\end{itemize}
The class $\Ec$ was introduced in \cite{ellis1976ghs} and it is also the class of $P$ for which the Lebowitz/GHS inequalities hold under the potential $P + \alpha \vp^2$ for every $\alpha \in \R$ \cite[Theorem 2]{newman1976rigorous}. We stress that a major difference compared to the Griffiths-Simon class is that it is \emph{easy to verify} whether a given potential belongs to $\Ec$.
\begin{figure}[h]
	\centering
	\includegraphics[scale=0.6]{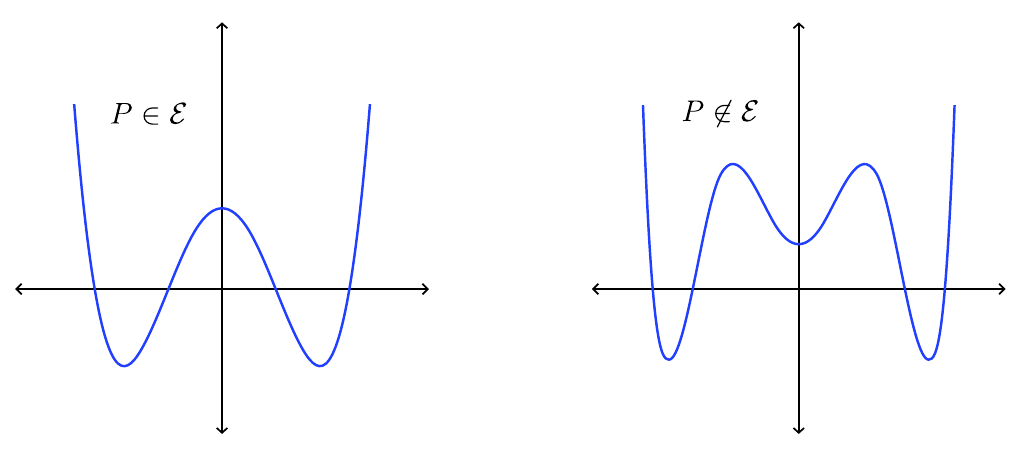}
	\caption{Left: a polynomial belonging to $\Ec$ since it has two wells (i.e.~local minima). Such potentials are expected in broad generality to exhibit behaviour consistent with the Ising model. Right: A polynomial which would not belong to $\Ec$ since it has three wells. Examples of such cases include some $\vp^6$ models (i.e.~when $P$ is a sextic polynomial) which are predicted to lie outside the Ising universality class. }
	\label{fig:doublewell}
\end{figure}
\begin{remark}
	It follows from the result mentioned above that the class $\Ec$ contains considerably many polynomials and other potentials which cannot belong the Griffiths-Simon class. As a concrete example, the potential 
	\[
	P(\vp) = \alpha \cosh(\vp) + \gamma \vp^2,
	\] 
	belongs to the class $\Ec$ for all values of $\alpha > 0$ and $\gamma \in \R$ but not to the GS class \cite{newman1976rigorous}. This model is the lattice version of an important quantum field theory called the sinh-Gordon model~\cite{SinhGordon,guillarmou20262d}.
\end{remark}

\subsection{Setup and result. }Fix $\Gamma$ to be an infinite (locally finite) transitive graph, a vertex $o \in \G$ and $J : \G \times \G \to \left[0,\infty\right)$ a family of translation-invariant coupling constants with
\begin{equation}\label{eq:Jasum}
	\sum_{x \in \G} J_{ox} < \infty.
\end{equation} 
We say that $J$ is \emph{finite-range} if there exists $L \geq 1$ such that $J_{xy}$ vanishes whenever $d(x,y) \geq L$.
We also fix an exhaustion of $\G$ by finite graphs $(G_n)_{n \geq 1}$ containing $o$.  
\begin{asum}\label{asum:two}
	We assume that $\G$ and $J$ are such that the weak limit
	\[
	\left< \cdot \right>_{\G,\beta,h} = \lim_{n \to \infty}  \left< \cdot \right>_{G_n,\beta,h} 
	\] 
	is well-defined, translation-invariant and
	\[
		\left< |\vp_o|^2 \right>_{\G,\beta,h} < \infty,
	\]
	for every $\beta \geq 0$, $h \geq 0$ and $P$ satisfying Assumption \ref{asum:one}.
\end{asum}
This assumption is known to be true when $\G = \Z^d$ for $d \geq 1$ \cite{ruelle1976probability, lebowitz1976statistical} and was recently extended to general $\G$ under a further decay assumption on $J$ \cite{panagiotis2026regularity}. 
The limit
\[
	\left< \cdot \right>^{+}_\beta = \lim_{h \downarrow 0}  \left< \cdot \right>_{\G,\beta,h} 
\] 
then exists by a standard monotonicity argument.
We define the \emph{susceptibility} to be
\[
	\chi(\beta) = \sum_{x \in \G} \left< \vp_o \vp_x \right>^+_{\beta}
\] 
which may or may not be finite. 
\begin{theorem}[]\label{thm:sharpness}
	Let $P$ belong the class $\Ec$. Then, there exists $\beta_c \in (0,\infty]$ such that
	\begin{itemize}
		\item (Disordered phase) if $\beta < \beta_c$, then there is finite susceptibility $\chi(\beta) < \infty$. Moreover, if $J$ is finite-range then there exists $c = c(\beta) > 0$ such that
			\[
			\left< \vp_o \vp_x \right>_{\beta}^+ \leq \exp( - c \, d(0,x)).
			\] 
		\item (Ordered phase) there is $c > 0$ such that for all $\beta \geq \beta_c$, the susceptibility is infinite and
			\begin{equation}\label{eq:exponent}
			\left< \vp_o \right>_{\beta}^+ \geq c  \,\frac{\beta - \beta_c}{\beta}.
			\end{equation} 
	\end{itemize}
\end{theorem}
Theorem \ref{thm:sharpness} implies the following more qualitative statement about Gibbs measures for the potential~$P$.
\begin{corollary}
The pure state $\left< \cdot \right>_{\beta}^+$ coincides with the free-boundary limit
\[
\left< \cdot \right>_{\beta}^0 = \lim_{n \to \infty} \left< \cdot \right>_{G_n,\beta,0},
\] 
for $\beta < \beta_c$ but for $\beta > \beta_c$, there is a symmetry breaking meaning that $\left< \cdot \right>_{\beta}^0 \neq \left< \cdot \right>_{\beta}^+.$
\end{corollary}
\begin{proof}
	By Theorem \ref{thm:sharpness}, we need only prove that $\chi(\beta) < \infty$ implies $\left< \cdot  \right>_{\beta}^{+} = \left< \cdot \right>_{\beta}^{0}$. The argument is classical. Since the GHS inequality is valid for $P \in \Ec$ \cite{ellis1976ghs}, we have that
	\[
	\left< \vp_o \right>_{G_n,\beta,h} - \left< \vp_o \right>_{G_n,\beta,0} \leq h \, \chi(\beta).
	\] 
	It follows that if $\chi(\beta) < \infty$ then, after taking the supremum over $n \geq 1$ and sending $h \to 0$, we have
	\begin{equation}
		\left< \vp_o \right>_{\beta}^{+} = \left< \vp_o \right>_{\beta}^{0}.
	\end{equation}
	Combined with the stochastic domination $\left< \cdot \right>_{\beta}^{+} \succeq \left< \cdot \right>_{\beta}^{0}$ \cite[Theorem 2.6.1]{simon2026phase} this is sufficient to prove that $\left< \cdot \right>_{\beta}^{+} = \left< \cdot \right>_{\beta}^{0}$. 
	
\end{proof}
\begin{remark} Again, the statement of the theorem deserves a few remarks. 
	\begin{itemize}	
\item The non-triviality of the phase transition (i.e.~$\beta_c < \infty$) is known to be true under quite general assumptions on the graph \cite{duminil2020existence,easo2025counting} and will not be discussed here. 
\item The exponent $\beta-\beta_c$ in \eqref{eq:exponent} does not, unlike the result of \cite{DCT}, match the mean-field exponent $(\beta-\beta_c)^{1 / 2}$ which is known for the Ising model when $d \geq 4$ \cite{AF} and we would expect to hold for each $P \in \Ec$.
\item Using the same method as in \cite[Lemma 2.8]{DCT}, our results also imply that
	\[
		\left< \vp_o \right>_{\beta_c,h}^{+} \gtrsim h^{1 / 2}
	\] 
	for $h \geq 0$ sufficiently small. Again, in this case the mean-field exponent should be $1 / 3$ rather than $1 / 2$. 
	\end{itemize}
\end{remark}
\subsection{Proof strategy. }We generalise the proof of Duminil-Copin and Tassion \cite{DCT}. The argument has two key inputs:

The first is a \emph{differential inequality} relating the magnetisation to a finitary quantity denoted $\vp_\beta(S)$. 
The inequality for the Ising model in \cite{DCT} (and similar differential inequalities in \cite{ABF}) was proved using the random current and its switching lemma. 
We prove an analogous inequality for any potential $P \in \Ec$ by utilising our new representation instead. The key step is an application of the FKG property of the representation. Even in the case of the Ising model, the argument itself is new. We also need an extra input (Lemma \ref{lemma:technical}) when the spins are continuous to show the constant appearing in the inequality does not degenerate. Whilst this second step may not be a surprising result we feel that the argument (based on stochastic domination) is quite novel.

The other input to the argument is an implication relating $\vp_\beta(S)$ with the model being in a disordered phase. This is proved using (a version of) the \emph{Simon-Lieb inequality} \cite{simon1980correlation,Lieb}, of which there are many different variants (see \cite{simon2026phase} for a historical discussion). 
The one that we need is contained in \cite{BFS,brydges1983random}. 
We stress that the validity of this inequality is a strong indicator that the phase transition is continuous, since it readily implies continuity of the mass gap. As such, this strategy is intrinsically related to the mechanism of \emph{continuous} phase transitions. 

\begin{remark}\label{remark:panagiotis}
	During the writing of this manuscript, Panagiotis and Veitch proved subcritical sharpness for any measure $\l$ (i.e.~without the restriction on $P \in \Ec$) when $J$ is finite-range \cite{panagiotis2026subcritical}. Their argument uses the machinery of the OSSS inequality~\cite{duminil2019sharp}, bypassing the incompatibility of the \cite{DCT} approach with first-order phase transitions. 
\end{remark}





\noindent\emph{Acknowledgements.}
LH and ML were supported by FWF Standalone grant ``Spins, Loops and Fields'' P 36298, and ML was supported by FWF SFB ``Discrete Random Structures'' 10.55776/F1002.

\section{The representation}\label{sec:therep}
\subsection{Summary. }The representation will associate to each spin system $\left< \cdot \right>_{G,\beta,h}$ a pair of coupled percolation models on $G$. The idea will be to start with two independent samples $\vp,\vp'$ and perform the Ginibre rotation leading to a coupled pair $(\vp^+,\vp^-)$.
The crucial observation is then that conditionally on both $|\vpp|$ and $|\vpm|$, the signs
\[
\sgn(\vpp) \quad\text{ and }\quad \sgn(\vpm)
\] 
are \emph{conditionally independent} ferromagnetic Ising models with $|\vp^{\pm}|$-dependent coupling constants.
Thus, they admit conditionally independent FK-Ising representations which we denote by $\w^+$ and $\w^-$.
The two percolations $\w^+$ and $\w^-$ have the property, essentially by construction, that applying the operations 
\[
	\vp^+ \to -\vp^+ \quad \text{ and }\quad \vp^- \to -\vp^-
\]
at every vertex in any cluster of $\w^+$ and $\w^-$ respectively is measure-preserving. 
\\[1em]
\noindent The representation has two particularly relevant cases, the Ising model and the Gaussian free field, which we describe in more detail below. For now, let us briefly mention that in these two cases the percolations $(\w^+,\w^-)$ are vertex-disjoint (for the Ising model) and independent (for the GFF) and this difference in behaviour can be seen as a manifestation of the fermionic and bosonic nature of the spin models (see Remark \ref{remark:wickandpfaff}).

\subsection{The representation. }We start with some notation. 
Let $G = (V,E)$, $J$, $h$ and $\l$ be as in Definition \ref{def:spinmodel}.
Given any subset $S \subset V$, we let $E(S)$ be the subset of edges in $E$ with both endpoints contained in $S$. We will often identify $S$ with the subgraph $(S,E(S))$ of $G$. We define
\[
\left( \vp, J \vp \right)_S = \sum_{xy \in E(S)} J_{xy} \vp_x \vp_y \quad\text{ and }\quad (h,\vp)_S = \sum_{x \in S} h_x \vp_x,
\] 
and write $(\vp,J\vp) = (\vp,J\vp)_V$ and $(h,\vp) = (h,\vp)_V$. 
We let $(\vp,\vp')$ be sampled from
\[
\left< \cdot \right>_{G,\beta,h} \otimes \left< \cdot \right>_{G,\beta,h}.
\]
The joint density of $(\vp,\vp')$ may now be written as
\[
	\exp\big( \beta (\vp,J\vp) + \beta (\vp',J\vp') + (h,\vp) + (h,\vp')\big) \d\l^V(\vp)\d\l^V(\vp'),
\] 
where we use the shorthand $\l^V = \otimes_{u \in V} \l$. 
In order to deal with a non-zero magnetic field, we will use a ghost vertex. We define the extended graph $ \overline{G} = ( \overline{V}, \overline{E})$ where 
\[
			\overline{V} = V \cup \left\{ \g \right\} \quad\text{ and }\quad \overline{E} = E \cup \left\{ v\g \mid v \in V \right\},
\]
and more generally we set $ \overline{S} = S \cup \left\{ \g \right\} $ and $ \overline{E}(S) = E \cup \left\{ v\g \mid v \in S \right\}$ for any $S \subset V$. 
\paragraph{A Ginibre rotation. }
Let $\vp$ and $\vp'$ be as above. 
We set
\begin{equation}\label{eq:repdef}
	\vp^+ = \frac{\vp + \vp'}{2} \quad\text{ and }\quad \vp^- = \frac{\vp - \vp'}{2},
\end{equation} 
which are the two coupled spin systems $\vp^+,\vp^-\in \R^V$ and their joint density may be written as 
\begin{equation}\label{eq:rotden}
	\exp( 2\beta (\vpp,J\vpp) + 2\beta (\vpm,J\vpm) + 2(h,\vpp)) \d\l^V(\vpp + \vpm)  \d\l^V(\vpp -  \vpm).
\end{equation} 
Note that the measure on $(\vpp,\vpm) \in \R^2$ with density $\d\l(\vpp + \vpm) \,\d\l(\vpp - \vpm)$ is invariant under each of the three maps:
\[
	(\vpp,\vpm) \to (-\vpp,\vpm), \quad (\vpp,\vpm) \to (\vpp,-\vpm), \quad (\vpp,\vpm) \to (\vpm,\vpp).
\] 
These are all simple consequences of the fact that $\l$ is symmetric. 

\paragraph{Two coupled FK-representations. }
From now on, we will set 
\[
\L^+ = |\vpp| \quad\text{ and }\quad \L^- = |\vpm|.
\]
The density \eqref{eq:rotden} relative to $\l^V(\vpp+\vpm)\otimes\l^V(\vpp-\vpm)$ can be written as
\[
	\exp\Big( 2\beta \sum_{xy \in E} J_{xy} \sgn(\vpp_x) \sgn(\vpp_y) \Lp_x \Lp_y + 2h\sum_{x \in V}h\sgn(\vp_x) \Lp_x + 2\beta \sum_{xy \in E} \sgn(\vpm_x) \sgn(\vpm_y) \Lm_x \Lm_y)\Big) 
\] 
By the symmetries of $\l(\vpp + \vpm) \otimes \l(\vpp - \vpm)$ discussed above, it is clear that the law of $\sgn(\vpp)$ and $\sgn(\vpm)$ conditional on $\Lp$ and $\Lm$ is that of two independent and ferromagnetic Ising models. Therefore they both admit typical (conditional) FK expansions which we now construct. 
\begin{definition}[]\label{def:def}
	Let $\P_{G,\beta,h}$ be the probability measure on $\R^V \times \R^V \times \left\{ 0,1 \right\}^{ \overline{E}} \times \left\{ 0,1 \right\}^E$ defining the law of $(\vpp,\vpm,\w^+,\w^-)$ in the following construction. 
	\begin{enumerate}[(1)]
	\item Sample $\vp$ and $\vp'$ from $\left< \cdot \right>_{G,\beta,h} \otimes \left< \cdot \right>_{G,\beta,h}$ and then define $\vp^+$ and $\vp^-$ as above. 
	\item We now construct the percolation model $\w^+ \in \left\{ 0,1 \right\}^{ \overline{E}}$. For each $uv \in E$, we set $\w^+_{uv} = 0$ (i.e.~close the edge) if $\vpp_u\vpp_v \leq 0$ and otherwise set $\w^+_{uv} = 1$ (i.e.~open the edge) with probability
		\[
			1 - \exp(-4\beta \xi^+_u \xi^+_v).
		\] 
		We also set $\w^+_{u\g} = 0$ if $\vpp_u \leq 0$ and otherwise set $\w^+_{u\g} = 1$ with probability $1-\exp(-4h\xi^+_u)$ for each $u \in V$. 
	\item We construct $\w^- \in \left\{ 0,1 \right\}^E$ via a similar procedure. That is, we open $\w^-_{uv}$ on an edge $uv \in E$ if and only if $\vp^-_u\vp^-_v > 0$ and then with probability $1 - \exp(-4\beta\xi^-_u\xi^-_v)$.
	\end{enumerate}
	The additional randomness used to determine if an edge of $\w^+$ or $\w^-$ is opened are all independent from one another.
\end{definition}
\noindent We will use the notation $\P^{+ / +}_\beta = \P_{G,\beta,h}$ and write $\P^{0 / 0}_\beta = \P_{G,\beta,0}$ for when there is no magnetic field. 
Note that an edge $\w^\pm_{uv}$ is $\P_{G,\beta,h}$-a.s.~only open if $\vp^\pm_u\vp^\pm_v > 0$. That is, the sign of $\vp^+$ and $\vp^-$ is constant on each cluster of $\w^+$ and $\w^-$ respectively, and moreover the sign of $\vp^+$ is always positive on the cluster of the ghost in $\w^+$.

The central property of this representation is the following.
\begin{proposition}[]\label{prop:signflip}
	Conditionally on $(\L^+,\L^-,\w^+,\w^-) \sim \P_{\beta}^{+ / +}$, the law of $(\vp^+,\vp^-)$ is invariant under flipping the sign of $\vp^\pm$ on any cluster of $\w^\pm$ (except for the cluster of the ghost in $\w^+$ for which the sign of $\vp^+$ is always fixed to be $+1$).
\end{proposition}
\begin{proof}
	As explained, conditional on $(\Lp,\Lm)$ the law of $\sgn(\vp^\pm)$ is coupled to $\w^\pm$ respectively exactly as in the FK-representation of the Ising model \cite{FK}. In particular, the conditional law of $\sgn(\vp^\pm)$ under $(\Lp,\Lm,\w^\pm)$ is invariant under flipping the sign on any (non-ghost) cluster of $\w^\pm$. Since they are conditionally independent, this extends to the joint law of both $\sgn(\vpp)$ and $\sgn(\vpm)$ under $(\Lp,\Lm,\w^+,\w^-)$. 
\end{proof}

\begin{remark}
	Since the map $\vp^- \to -\vp^-$ is equivalent to exchanging $\vp$ and $\vp'$, we see that the law of $(\vp,\vp')$ is invariant under swapping $\vp$ and $\vp'$ on any cluster of $\w^-$. It is for this property (termed `cluster swapping') that Sheffield introduced $\w^-$ in \cite{She}.
\end{remark}

\noindent 
 \begin{remark}\label{remark:extensions}
	The coupling can be extended in many ways and one possibility is to couple different values of the magnetic field. Indeed, we may take
		\[
					(\vp , \vp') \sim \left< \cdot \right>_{G,\beta,h} \otimes \left< \cdot \right>_{G,\beta,0},
		\]
		where $h>0$ and proceed as before. We leave of the details of defining the coupling, which we denote by $\P^{+ / 0}_\beta$, to the reader other than to note that in this case both $\w^+$ and $\w^-$ live on the extended graph and the sign of $\vp^\pm$ is always $+1$ on the $\w^\pm$ cluster of the ghost.
\end{remark}
\paragraph{The $\I$ function. }
Although not strictly necessary for the definition of the representation, we now introduce some notation which will be useful in Part \ref{part:2} and also help with developing an intuition for the behaviour for different choices of $P$.
Suppose that $P$ is a continuous potential on an interval $\left[ -A,A \right]$ for some $0 < A \leq \infty$, with $P = \infty$ outside $\left[ -A,A \right]$. In this case, we denote by
\[
	\l_2 = \l^{(2P)}
\] 
and make the following important definition.	
\begin{definition}[]
	We define the function $\I : \R \times \R \to \R \cup \left\{ \infty \right\} $ given by 
\begin{equation}\label{eq:Idef}
	\I(x,y) = 
\begin{cases}
	P(x+y) + P(x-y) - 2 P(x) -  2 P(y) \quad & \text{ if } |x| + |y| \leq A,\\ 
	\infty \quad & \text{ otherwise,}
\end{cases}
\end{equation} 
so that{\hypersetup{linkcolor=black}\footnote{The choice to write the measure relative to $\d\l_2 \otimes \d\l_2$ rather than $\d\l \otimes \d\l$ is arbitrary for now but will be useful later on.}}
\[
	(\d\l \otimes \d\l)(\vp,\vp') = \exp( -\I(\vpp,\vpm)) (\d\l_2 \otimes \d\l_2)(\vpp,\vpm).
\]
\end{definition}
\noindent The important properties of $\I$ are that it is symmetric in $x$ and $y$ and only depends on $|x|$ and $|y|$. Thus, \eqref{eq:rotden} may be rewritten as 
\begin{equation}\label{eq:rotden2}
	\exp( 2\beta (\vpp,J\vpp) + 2\beta (\vpm,J\vpm) + (2h,\vpp)) \exp\big( -\textstyle\sum_{u \in V} \I(|\vpp_u|,|\vpm_u|)\big) \d\l_2^G(\vpp) \d\l_2^G(\vpm).
\end{equation}
We view $\I$ as \emph{capturing all of the interaction} between $\vpp$ and $\vpm$ in \eqref{eq:rotden2}.
In particular, one can think of $(\vpp,\vpm)$ as two initially independent spin systems which are coupled together at each site $u \in V$ by the value of $\I(|\vpp_u|,|\vpm_u|)$. 
\subsection{Correlation functions.}
As a first consequence, we have expressions for the standard correlation functions, 
\begin{equation}\label{eq:corr1}
	\left< \vp_x \vp_y \right>_{\beta, h} = 2\E^{+ / +}_\beta [ \Lp_x \Lp_y \i( x \overset{\w^+}{\longleftrightarrow} y) ],
\end{equation} 
where as usual $x \overset{\w^\pm}{\longleftrightarrow} y$ means that $x$ and $y$ belong to the same connected component of $\w^\pm$, and similar expression for higher order correlations.
The real advantage is that we may also express \emph{truncated} correlations, such as 
\begin{equation}\label{eq:corr2}
	\left< \vp_x \vp_y \right>_{\beta,h} - \left< \vp_x \right>_{\beta,h} \left< \vp_y \right>_{\beta,h} = 2\E^{+ / +}_\beta[ \Lm_x \Lm_y \i(x \overset{\w^-}{\longleftrightarrow} y) ] 
\end{equation}
for $x,y \in V$. 
Another example (which we will use later) is
\begin{equation}\label{eq:corr3}
	\left< \vp_x \vp_u \vp_v \right>_{\beta,h} - \left< \vp_x \right>_{\beta,h} \left<\vp_u \vp_v \right>_{\beta,h} = 2\E^{+ / +}_\beta[ \Lm_x\Lm_u\Lp_v \i( x \overset{\w^-}{\longleftrightarrow} u, v \overset{\w^+}{\longleftrightarrow} \g )  ] + (u \leftrightarrow  v),
\end{equation}
for $x,u,v \in V$, where the last term denotes the previous term with the roles of $u$ and $v$ swapped.
One can also compute differences in correlations between two measures (using Remark \ref{remark:extensions}), such as
\begin{equation}\label{eq:corr4}
	\left< \vp_x \vp_y \right>_{\beta,h} - \left< \vp_x \vp_y \right>_{\beta,0} = 2 \E^{+ / 0}_\beta[ \Lp_x\Lm_y \i(x \overset{\w^+}{\longleftrightarrow} \g, y \overset{\w^-}{\longleftrightarrow} \g) ].
\end{equation} 
All these expression are reminiscent of similar identities for Ising correlations in terms of random currents.
\begin{proof}[Proof of \eqref{eq:corr1}--\eqref{eq:corr4}]
	Each may be derived in the same manner by first reexpressing the left-hand side in terms of joint correlations for $\vp$ and $\vp'$, and then in terms of $\vpp$ and $\vpm$ before applying Proposition \ref{prop:signflip}.
	As one example,
	\begin{equation}\label{eq:corrdev}
		\left< \vp_x \vp_y \right>_{\beta,h}^{} - \left< \vp_x \right>_{\beta,h}^{} \left< \vp_y \right>_{\beta,h}^{}  = 2\E^{+ / +}_\beta[ \vp^-_x \vp^-_y ] 
		= 2\E^{+ / +}_\beta[ \Lm_x\Lm_y \i(x \overset{\w^-}{\longleftrightarrow} y) ],
	\end{equation}
	where the second equality is just \eqref{eq:repdef} and the final inequality is Proposition \ref{prop:signflip}, which gives exactly \eqref{eq:corr2}. Similarly, we can derive \eqref{eq:corr3} by using the identity
	\[
		\vp_x( \vp_u \vp_v - \vp'_u\vp'_v) = 2(\vpp_x + \vpm_x)(\vpm_u\vpp_v + \vpp_u\vpm_v),
	\]
	noting that any terms with just one $\vpm$ factor may be dropped since $(\vpp,\vpm)$ and $(\vpp,-\vpm)$ are equidistributed under $\P^{+ / +}_\beta$ and then proceeding as in \eqref{eq:corrdev}. 
\end{proof}
\noindent\emph{Example: The Ising model.} In fact, the identities are of interest even in this case. 
Let $\sigma$ and $\sigma'$ (instead of $\vp$ and $\vp'$) be two independent samples from an Ising model. Then, we can alternatively write $\vp^\pm_u = (\sigma_u \pm \sigma'_u) / 2$ as
\begin{equation}\label{eq:taurole}
\vp^\pm_u = \sigma_u \i( \tau_u = \pm 1), 
\end{equation} 
where $\tau = \sigma \sigma'$ is the product, which is the so-called XOR-Ising model. In particular,
\[
1 - \exp(-4\beta \xi^\pm_u \xi^\pm_v) = \begin{cases}
	1 - \exp(-4\beta) \quad & \text{ if }\tau_u = \tau_v = \pm 1\\
	0 \quad & \text{ otherwise,}
\end{cases}
\]
and $\vp^\pm_u\vp^\pm_v > 0$ if and only if both $\tau_u = \tau_v = \pm 1$ and $\sigma_u = \sigma_v$.
Therefore, Definition \ref{def:def} coincides with the definition of $\w$ in \cite[Definition 2.6]{AHL}.
It follows from \eqref{eq:taurole} that if $\l$ is supported on $\left\{ \pm 1 \right\}$, then always either
\begin{equation}\label{eq:isingext}
	\vp^+_x = 0 \quad \text{ or }\quad \vp^-_x = 0,
\end{equation} 
for every vertex $x \in V$. This leads to a great simplification in the identities above since either $|\vp_x^\pm|$ is exactly $1$ or $x$ is incident to no open edge of $\w^\pm$. Moreover, no vertex is incident to an edge of both $\w^+$ and $\w^-$ so we may treat $\w^+$ and $\w^-$ as a partition of the connected components of a single percolation model $\w = \w^+ \cup \w^-$. Using this, \eqref{eq:corr1} reduces to
\[
	\left< \sigma_x \sigma_y \right>_{\beta,0} = \P^{0 / 0}_\beta\big(x \overset{\w}{\longleftrightarrow} y\big), 
\] 
and more generally $\w$ shares the same Edwards-Sokal property as the FK-Ising model \cite[Theorem 2.7]{AHL}. Similarly, \eqref{eq:corr2} gives a simple expression for the truncated correlation{\hypersetup{linkcolor=black}\footnote{One which is arguably simpler than the random current representation for the same quantity where the relevant (dis)connection event is under a measure which depends on $x$ and $y$}} 
\begin{equation}\label{eq:isingtrunc}
	\left< \sigma_x; \sigma_y \right>_{\beta,h} = 2\P^{+ / +}_\beta\big(  x \overset{\w^-}{\longleftrightarrow} y \big),
\end{equation}
We refer to \cite{AHL} for a greater discussion on this case.
In the planar setting, it turns out that the law of $\w$ is dual to the \emph{double random current} \cite[Proposition 2.8]{AHL}, 
and it is closely related to many representations of the six-vertex model \cite{GlaPel,Lis22b,glazman2025delocalisation}. 
The percolation $\w^-$ and the identities \eqref{eq:corr1} and \eqref{eq:corr2} appeared in the disagreement percolation approach of \cite{aizenman2020exponential}.
\\[1em]
\noindent\emph{Example: The $\vp^4$ model.} Another important example is when $P$ is a quartic even polynomial
\[
P(\vp) = \alpha \vp^2 + g \vp^4
\] 
with $g > 0$. In this case, the function $\I$ is given by 
\begin{equation}\label{eq:phi4I}
\I(x,y) = 12g x^2y^2.
\end{equation} 
It is not hard to guess that because of \eqref{eq:phi4I}, the reweighting by $\exp(-\I)$ in \eqref{eq:rotden2} will lead to the pair $\Lp$ and $\Lm$ being negatively associated with each other. This is a special case of Proposition \ref{prop:FKG} which holds for any potential $P \in \Ec$. We stress that this representation is distinct to those appearing in \cite{GPPS,GPPS2}.

In the degenerate example where $P$ is only quadratic, the function $\I$ vanishes and there is no interaction between $\vpp$ and $\vpm$ owing to the fact that the rotation of two independent Gaussians preserves said independence. The $\vp^4$ model thus interpolates between two extremal cases of Proposition~\ref{prop:FKG}: the Gaussian limit where $\vp^+$ and $\vp^-$ are independent and the Ising limit with the hard-core constraint \eqref{eq:isingext}.

\subsection{The FKG inequality.} The FKG inequality is an important correlation inequality which has come to play a major role in statistical mechanics. The usefulness of Definition~\ref{def:def} stems in part from the FKG inequality stated below. Proposition \ref{prop:FKG} was first proved for the Ising model in \cite[Proposition 3.1]{AHL}. The proof of Proposition \ref{prop:FKG} may be found in Section \ref{sec:corrineq} and is based on the ideas of \cite[Theorem 2.8]{LamOtt}. 
\begin{proposition}[]\label{prop:FKG}
	Let $P \in \Ec$. Then, the quadruple 
	\[
		(\L^+,-\L^-,\w^+,-\w^-)
	\] 
	satisfies the FKG inequality under $\P_{G,\beta, h}$. In particular, $\w^+$ and $\w^-$ are negatively associated with each other. 
\end{proposition}
See Section \ref{sec:corrineq} for a precise statement and introduction to the theory of the FKG inequality. The content of Proposition \ref{prop:FKG} is that individually $(\L^+,\w^+)$ and $(\L^-,\w^-)$ are positively associated, but they are \emph{negatively associated} with each other{\hypersetup{linkcolor=black}\footnote{We point out that quite differently to the rotated system, the absolute value of just a single unrotated spin $|\vp|$ with $\vp \sim \left< \cdot \right>_{G,\beta,h,P}$ satisfies the FKG inequality for any potential $P$.}}.
We also record an immediate corollary of this result, which may be of independent interest.
\begin{corollary}[]
	Let $P \in \Ec$ and $(\vp,\vp') \sim \left< \cdot \right>_{G,\beta,h} \otimes \left< \cdot  \right>_{G,\beta ,h}$. Then, the pointwise product
	\[
	\vp \times \vp' \in \R^V
	\] 
	satisfies the FKG inequality.
\end{corollary}
\begin{proof}
	We have $\vp\vp'/4  =  |\vpp|^2 - |\vpm|^2$, which is an increasing function of the pair $(|\vpp|,-|\vpm|)$.
\end{proof}
\noindent As a first application of this FKG inequality, we have that
\begin{equation}\label{eq:lebowpm1}
\E[ \L^+_x \L^+_y \L^-_z \L^-_w \i(x \overset{\w^+}{\longleftrightarrow} y, z \overset{\w^-}{\longleftrightarrow} w) ] \leq \E[ \L^+_x \L^+_y \i(x \overset{\w^+}{\longleftrightarrow} y )] \E [ \L^-_z \L^-_w \i(z \overset{\w^-}{\longleftrightarrow} w) ],
\end{equation} 
or equivalently,
\begin{equation}\label{eq:lebowpm}
\E[ \vp^+_x \vp^+_y \vp^-_z \vp^-_w ] \leq \E[ \vp^+_x \vp^+_y ] \E[ \vp^-_z \vp^-_w ].
\end{equation} 
for any $x,y,z,w \in V$. Rearranging \eqref{eq:lebowpm} gives 
\begin{equation}\label{eq:lebow}
	\left< \vp_x \vp_y \vp_z \vp_w \right> \leq \left< \vp_x \vp_y  \right> \left< \vp_z \vp_w \right> + \left<\vp_x \vp_z \right> \left<\vp_y \vp_w \right> + \left<\vp_x \vp_w \right>\left<\vp_y \vp_z \right>
\end{equation}
which is the classical \emph{Lebowitz inequality}. This is equivalent to the \emph{GHS inequality}, i.e.~the concavity of the map 
\[
h \mapsto \left< \vp_u \right>_{G,\beta,h}
\] 
for $h \geq 0$ and $u \in G$. 
A proof of \eqref{eq:lebowpm} and so the Lebowitz inequality for $P \in \Ec$ 
can also be deduced by other means, see \cite{simon2026phase} for a historic overview.
As mentioned above, the converse to this result is that if the Lebowitz inequality holds for the potential $P + \alpha \vp^2$ for all $\alpha \in \R$ then $P \in \Ec$ \cite[Theorem 2]{newman1976rigorous}. 
Many other correlation inequalities, for instance
\[
\left< \vp_x \vp_y \right>_{\beta,h}^{} - \left< \vp_x \vp_y \right>_{\beta,0}^{} \leq \left< \vp_x \right>_{\beta,h} \left< \vp_y  \right>_{\beta,h}^{} \quad \text{ for }h \geq 0 \text{ and }P \in \Ec,
\] 
can be proven using double cluster swapping. One can also prove the MMS inequalities \cite{MMS} using a variant of the representation where $\vp \sim \left< \cdot \right>_{G,\beta,0}^{}$ and $\vp'$ is not sampled independently but instead defined by $\vp'_u = \vp_{Ru}$ for a suitable `reflection' $R$. The rest of the argument is left to the reader. 
\begin{remark}
	The definition and results of this section may be easily generalised to the case where the potential $P = P_v$ or the magnetic field $h = h_v$ varies across the vertices of $v \in V$. For instance, Proposition~\ref{prop:FKG} continues to hold so long as each $P_v \in \Ec$ and $h_v \geq 0$.
\end{remark}

\subsection{The representation in infinite-volume. }In this work we will only utilise this representation on a finite graph yet we briefly note there are simple infinite-volume versions of these identities which may shed light on various aspects of the phase transition. As one example, 
\begin{equation}\label{eq:infvolcor}
	\left< \vp_x \vp_y \right>^{+} -\left< \vp_x \vp_y \right>^{0} = 2 \E^{+ / 0}_\beta[ \L^+_x \L^-_y \i(  x \overset{\w^+}{\longleftrightarrow} \infty , \, y \overset{\w^-}{\longleftrightarrow} \infty   )].
\end{equation} 
The vanishing of the left hand side has been proved for the Ising and $\vp^4$ models at any temperature and on any transitive amenable graph of polynomial growth \cites{Rao17,GPPS}, leading to the statement that
\[
	\left< \cdot \right>^0_\beta = \tfrac{1}{2}\left< \cdot \right>^+_\beta + \tfrac{1}{2}\left< \cdot \right>^{-}_\beta.
\]
We expect this to hold more generally for any potential $P \in \Ec$, which in view of \eqref{eq:infvolcor} would corresponds to a \emph{non-coexistence result} for the two negatively associated percolation processes $\w^+$ and $\w^-$. 
So far we were not able to prove this.

\part{Pfaffian relations}\label{part:1} 

We adopt the setup in Section \ref{sec:app2} and, in particular, denote by $ \left< \cdot \right>_{G} $ the spin model defined in Definition \ref{def:spinmodel} with no external field and with reference measures $(\l_v)_{v\in V}$, and by $\P$ the representation defined in Section \ref{sec:therep}. We make the harmless assumption that $J_{uv} > 0$ for any $uv \in E$.
For any subset $A \subset V$ and any $\vp \in \R^V$, we define $\vp(A) = \prod_{v \in A} \vp_v$.
The starting point for this section is the following identity (a simple generalisation of \eqref{eq:corr2}): For any two subsets of vertices $A,B \subset V$, we have
\begin{equation}\label{eq:corrpm}
\E[ \vpp(A) \vpm(B) ] = \E [ \L^+(A) \L^-(B) \i( \F^+_A\cap  \F^-_B) ],
\end{equation}
where $ \F^\pm_A $ is the event that the connected components of the percolation $\w^\pm$ induce an even partition of the vertices of $A$. 

\subsection{An algebraic characterisation. }We start with a result of \cite{Lis22a}. Let $W$ be a subset of vertices equipped with an ordering. We imagine placing the vertices of $W$ on the boundary of the unit disk in a cyclic manner with respect to this order. A pairing (i.e.~a partition into sets of size two) of an even subset $W$ may be represented by continuous simple curves contained in the disk (called arcs) which connected the paired vertices. Let $(A,B)$ be a pair of disjoint even subsets $A,B \subset W$. It is said to be \emph{planar} if there exists a pairing of $A$ and of $B$ such that no arc of the first pairing intersects an arc of the second pairing, and otherwise is said to be \emph{non-planar}.  

\begin{lemma} \label{lem:Marcin_Pfaf}
    A spin model $\left< \cdot \right>_{G}$ satisfies Pfaffian relations 
    on $W \subset V$ if and only if for each non-planar pair $(A,B)$, we have
    \begin{equation}\label{eq:ABeq}
	    \E [ \vp^+(A) \vp^-(B) ] = 0.
    \end{equation}
\end{lemma}
\begin{proof}
	Recall that for any subset $I \subset W$, $M_I$ is the minor of the $n \times n$ matrix $M_{ij} = \left< \vp_{v_i} \vp_{v_j} \right>$ where $v_1,\ldots,v_n$ are the (ordered) vertices of $W$. It is an algebraic property of Pfaffians that for any non-planar subset $(A,B)$, the sum
	\begin{equation}\label{eq:Pfinduc}
		\sum_{ I \subseteq A \cup B} (-1)^{|I \cap A|} \Pf(M_{I}) \Pf(M_{( A \cup B) \setminus I})
	\end{equation} 
	vanishes \cite[Corollary 9]{Lis22a}. One can argue by induction that equations \eqref{eq:Pfinduc} recursively generate higher and higher-order Pfaffians so that the relations are equivalent to 
	\begin{equation}\label{eq:Spinduc}
		\sum_{ I \subseteq A \cup B} (-1)^{|I \cap A|} \left<\vp(I)  \right>_G \left<\vp((A \cup B) \setminus I) \right>_G = 0
	\end{equation} 
	for each non-planar subset $(A,B)$. The equations \eqref{eq:Spinduc} may be simply rewritten as 
	\[
	\E[ \vp^+(A) \vp^-(B) ] = 0
	\]
	for each non-planar subset $(A,B)$, proving the equivalence. 
\end{proof}
\begin{remark}\label{remark:minimal}
A weaker sufficient condition for Pfaffian relations (that we will not use) is the vanishing of \eqref{eq:ABeq} for the family of non-planar pairs $(A,B)$ with $|A| = 2$, or, more generally, any family of non-planar pairs $(A,B)$ large enough to contain at least one pair $(A,B)$ with $A \cup B = W'$ for each subset $W' \subset W$. 
\end{remark}
We may also obtain a characterisation solely in terms of $\w^+$ and $\w^-$ which will be more useful in understanding the geometry of $G$.
\begin{corollary} \label{cor:pfaff}
    A classical spin model satisfies Pfaffian relations on $W \subset V$ if and only if for every non-planar pair $(A,B)$, 
    \begin{equation}\label{eq:pfaffchar}
        \mathbf{P}(\F_A^+ \cap \F^-_{B}) = 0. 
    \end{equation}
\end{corollary}
\begin{proof}
	Using \eqref{eq:corrpm} and the previous lemma, we need only show that 
	\[
		\P \left( \F_A^+ \cap \F_{B}^- \right) > 0 \implies \E\left[ \L^+(A) \L^-(B) \i\left( \F^+_A \cap \F^-_B \right)  \right] > 0
	\] 
	since the converse is clear.	However, this is rather immediate given that the event $\w^\pm \in \F_A$ requires that each vertex $v \in A$ is not an isolated component of $\w^\pm$ and so must have $\L^\pm_v > 0$. 
\end{proof}
\noindent\emph{Example: the Ising model. }We may use Corollary \ref{cor:pfaff} to easily prove the relations in the context of the planar Ising model. There are just two key observations:
\begin{itemize}
	\item 
Let $G$ be planar and suppose that each potential $P_v$ is of Ising type, then the percolations $\w^+$ and $\w^-$ \emph{cannot cross each other.} Indeed, any vertex $v \in V$ with $ |\vp^\pm_v| = 0$ can have no incident open edges in $\w^\pm$ and yet it must satisfy \eqref{eq:isingext}.
	\item 
Suppose now that $W$ is genuinely a subset of the boundary vertices on the planar graph $G$ (i.e.~each vertex of $W$ is incident to a common face) equipped with its cyclic ordering. Then, almost by definition, for any non-planar pair $(A,B)$, the event 
\[
	\F^+_A \cap \F^-_{B}
\] 
topologically forces $\w^+$ and $\w^-$ to cross each other and so must have probability zero under $\P$, ensuring \eqref{eq:pfaffchar}. 
\end{itemize}
An argument along these lines (except using a random current representation) appeared in \cite{Lis22a}.
\begin{remark}\label{remark:wickandpfaff}
	The bosonic counterpart to the Pfaffian relations is \emph{Wick's rule}:
	\[
		\Big< \prod_{i \in I} \vp_i \Big>_{G} = \textrm{Hf}(M_I) \quad \text{ for every $I \subset W$},
	\] 
	where $\textrm{Hf}$ is the \emph{hafnian}, which has the same diagrammatic expansions as the Pfaffian (e.g.~Figure \ref{fig:fourpt}) except each term is now unsigned. 
	It is not difficult to see that there is a similar condition to Lemma \ref{lem:Marcin_Pfaf}: Wick's rule is satisfied if and only if
	\[
		\E [ \vp^+(A) \vp^-({B}) ]  = \E [ \vp^+(A) ] \E[ \vp^-(B) ] 
	\] 
	for every two disjoint subsets $A,B$ of $V$. This relation clearly holds when $\vp$ is the Gaussian free field (e.g.~when $P$ is a quadratic even polynomial) as it is a classical fact that $\vpp$ and $\vpm$ are then independent. 
	Let us now consider the class $\Ec$ introduced in Section \ref{sec:app2}. Under the assumption $P \in \Ec$, we have that (arguing similarly to \eqref{eq:lebowpm1})
	\[
		0 \leq \E [ \vp^+(A) \vp^-({B}) ] \leq \E[ \vp^+(A) ] \E[ \vp^-({B}) ].
	\] 
	Interestingly, we see that the two saturating cases of these inequalities can be viewed as arising from the Pfaffian and Wickian relations for fermions and bosons respectively. 
\end{remark}

\subsection{A geometric characterisation. }
Let $\Vis$ be the set of vertices $v \in V$ which are of Ising type, i.e.~the reference measure $\l$ is uniform on $\left\{ -s,s \right\}$ for some $s > 0$.

\begin{definition}\label{def:Wcross}
    Let $G = (V, E)$ be a finite graph, with $W = (v_1, \ldots, v_{n})$ a tuple of distinct vertices in $V$, and let $A\subset V$. 
    We say that $G$ satisfies the $(W,A)$-crossing property if for every $1 \leq i < j < k < l \leq n$, each pair of paths in $G$ connecting $v_i$ with $v_k$ and $v_j$ with $v_l$ must intersect at a vertex of $A$. 
\end{definition}
We say that $G$ satisfies the $W$-crossing property if it satisfies the $(W,V)$-crossing property.

\begin{lemma}\label{lemma:keylemma}
    The spin model $\left< \cdot \right>_G$ satisfies Pfaffian relations on $W = \{v_1,v_2,\ldots, v_n \}$ if and only if $G$ satisfies the $(W,\Vis)$-crossing property.
\end{lemma}
\begin{proof}
    Fix $1 \leq i < j < k < l \leq n$ as in the lemma. 
    Suppose there exists a pair of paths $\gamma, \tilde \gamma$ connecting $v_i$ to $v_k$ and $v_j$ to $v_l$ respectively, which do not intersect at a vertex of $V_{\mathrm{Ising}}$. 
    We claim that in this case
    \begin{equation}\label{eq:nonvanish}
	    \E [ \vp^+_{v_i}\vp^+_{v_k} \vp^-_{v_j} \vp^-_{v_l} ] > 0, 
    \end{equation}
    which contradicts the Pfaffian relations by Lemma \ref{lem:Marcin_Pfaf}, since the pair $(\left\{v_i,v_k \right\},\left\{ v_j,v_l \right\})$ is non-planar.
   To see this is true, notice that as long as such paths exist, 
    the probability that $\gamma \subseteq \omega^-$ and $\tilde \gamma \subseteq \omega^+$ is strictly non-zero (conditional on the values of spins outside the paths, the probability that the paths are open in $\omega^+$ and $\omega^-$ is not zero). 
    Using \eqref{eq:corrpm}, we have
    \[
	    \E\big [ \vp^+_{v_i}\vp^+_{v_k} \vp^-_{v_j} \vp^-_{v_l} \big]
        = \E \big[ |\vp^+_{v_i}\vp^+_{v_k} \vp^-_{v_j} \vp^-_{v_k} 
        |\id (v_i \con{\omega^+} v_k,v_j \con{\omega^-} v_l )  \big], 
    \]
    which is strictly positive by the above observation. The equivalence should now be clear. 
\end{proof}
\begin{remark}
	Note that the proof of Lemma \ref{lemma:keylemma} only used Lemma \ref{cor:pfaff} for the case of non-planar pairs $(A,B)$ with $|A| = |B| = 2$. This immediately proves Corollary \ref{cor:thm1}. In general, this is of course not a sufficient algebraic characterisation (see Remark \ref{remark:minimal}) so it is somewhat surprising that in this setting of spin models it is. 
\end{remark}

\begin{figure}[H]
	\centering
	\includegraphics[scale=0.6]{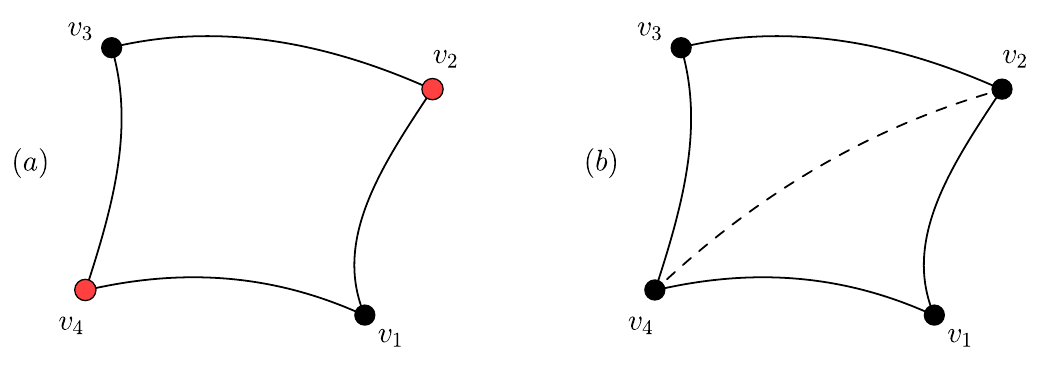}
	\caption{}
	\label{fig:counterexample}
\end{figure}

\begin{remark}\label{remark:expexamp}
	We now give a curious example showing that it must not necessarily be the case that each vertex in $W$ is of Ising type. 
	Suppose that $G$ is as depicted in Figure \ref{fig:counterexample} (a).
	We suppose that the two vertices in black are of Ising type, but the potentials at the two in red are unrestricted. Then, the only non-trivial Pfaffian relation on $\left\{ v_1,v_2,v_3,v_4 \right\}$, namely 
\[
	\left<\vp_{v_1} \vp_{v_2} \vp_{v_3} \vp_{v_4} \right>_G = \left<\vp_{v_1} \vp_{v_2} \right>_G \left< \vp_{v_3} \vp_{v_4} \right>_G + \left< \vp_{v_1} \vp_{v_4} \right>_G \left< \vp_{v_2} \vp_{v_3} \right>_G - \left< \vp_{v_1} \vp_{v_3} \right>_G \left<\vp_{v_2} \vp_{v_4} \right>_G.
\] 
	is equivalent to (using Lemma \ref{lemma:keylemma}) the property that any two paths, one from $v_1$ to $v_3$ and one from $v_2$ to $v_4$, must intersect at a vertex of Ising type (i.e.~a black vertex). This is certainly the case and so no matter the potential on $v_2$ and $v_4$ the Pfaffian relations are satisfied. 
	If $G$ is instead the graph depicted in Figure \ref{fig:counterexample} (b) then it is essential that both $v_2$ and $v_4$ are of Ising type for the relations to hold.
\end{remark}
\section{Graph reductions}\label{sec:graphred}
Let $G = (V, E)$ be a finite connected graph with a distinguished subset $W$ of vertices equipped with some ordering. 
A path in $G$ is a finite sequence of vertices in $V$ with the property that each two 
consecutive vertices are either the same vertex or are connected by an edge in $E$.
Let $S$ be a nonempty subset of $V$. 
We say that $S$ separates one subset $A \subset V$ from another subset  $B \subset V$ (or $S$ is an \emph{$(A, B)$-separator}) in $G$ 
if every path connecting $A$ and $B$ in $G$ contains a vertex of $S$. 
We call $S$ a \emph{reduction set} (for $G$) if:
\begin{itemize}
    \item $|S|\leq 3$,
    \item $S$ separates $W$ from a nonempty set $X \subset V$ with $X \cap S = \emptyset$,
    \item $S$ is \emph{minimal} in the sense that for each $T \subsetneq S$, 
    $T$ is not a $(W, X)$-separator for any $X$ as above.
\end{itemize}
We will say that $G$ is \emph{fully reduced} if there does not exist a reduction set for $G$.
\\[1em]
\noindent Given any set $S$ (not necessarily a reduction set), we construct the \emph{reduced graph} $R({G},S)$ by deleting from $G$ 
all vertices in $V\setminus S$ separated from $W$ by $S$, 
and adding an edge between each pair of vertices of $S$ (so none if $|S| = 1$). 
Note that the set $W$ is by definition always included in the set of vertices of $R({G},S)$, 
which in turn is a strict subset of~$V$. This allows us to keep the same distinguished set $W$ 
of vertices when applying consecutive reductions to the original graph.

\subsection{Spin model reduction. }
The goal of this section is to explain how certain sets $S$ allow for the spin model on $G$ to be transformed into a new spin model on $R(G,S)$ whilst preserving the boundary correlation functions. There are two types of moves we will consider, both of which transform $G \to R(G,S)$ but require different restrictions on the set $S$:
\begin{itemize}
	\item (Move 1) $S$ is a reduction set which is contained in $\Vis$ (Lemma \ref{lemma:firstmove}).
\end{itemize}
The second move is only relevant in the situation where $W \not\subset \Vis$ (which as we saw in Remark~\ref{remark:expexamp} can happen). It weakens the requirement on $S$ being a subset of Ising-type vertices, but requires that any vertex in $S \setminus \Vis$ must in the reduced graph be adjacent only to other vertices of $S$. In practice, we will only apply this move when $S$ is a subset of $\Vis \cup W$. 
\begin{itemize}
	\item (Move 2)   $S \subset V$, $|S| \leq 3$ and any vertex $v \in S \setminus \Vis$ is only adjacent to vertices separated from $W$ by $S$ (Lemma \ref{lemma:secondmove}).
\end{itemize}
Technically, the first move falls under the umbrella of the second move but we will treat them separately since the first is slightly easier and contains most of the main ideas.
\\[1em]
In this section, we assume without loss of generality (by rescaling the coupling constants) that for each $v \in \Vis$ the measure $\l_v$ is uniform on the set $\left\{ \pm 1 \right\}$ and also that $\beta = 1$. 
\\[1em]
\noindent Our main tool to transform the spin model is the following lemma. It asserts that the correlations between a fixed subset $H \subset V$ of spins in $\left< \cdot \right>_{G}$ may be constructed as the appropriately rescaled correlations of an Ising model on $|H|$ vertices, {so long as $|H| \leq 3$}. Its proof is postponed to the end of this section. Note that no such result can be true if $|H| \geq 4$ as demonstrated by the Lebowitz inequality \eqref{eq:lebow} which is not true for every potential \cite{simon1973varphi4}.  

\begin{lemma}\label{lemma:isingreplacement}
	Let $H \subset V$ with $|H| \leq 3$. Then, there exists $ K_{uu'} \geq 0$ for every $u,u' \in H$, $r_{u} > 0$ for each $u \in H$ and a constant $Z > 0$ such that
	\begin{equation}\label{eq:goalcorr3}
		\left< \vp(A) \right>_{G} = \Big( \prod_{u \in A} r_u  \Big) \frac{1}{Z}\sum_{ \sigma \in \left\{ \pm 1 \right\}^H} \sigma(A) \exp\Big( \sum_{u,u' \in H} K_{uu'} \sigma_u \sigma_{u'} \Big) \quad \text{ for any $A \subset H$.}
	\end{equation} 
	Moreover, we may choose $r_u = 1$ for any $u \in H \cap \Vis$.
\end{lemma}
The importance of choosing $r_u = 1$ for any $u \in H \cap \Vis$ is that it ensures that the distribution of $\vp_{H \cap \Vis} \in \left\{ \pm 1 \right\}^{H \cap \Vis}$ under $\left< \cdot \right>_{G}$ matches the distribution of $\sigma \in \left\{ \pm 1 \right\}^{H \cap \Vis}$ under the Ising measure in \eqref{eq:goalcorr3}.
We start with the transformation for the first type of move.  
\begin{lemma} \label{lemma:firstmove}
    Let $S$ be a reduction set for $G$ such that $S \subset \Vis$. 
    Then, there exists a classical spin model $ \widetilde{\vp}$ on the reduced graph $ \widetilde{G} = ( \widetilde{V}, \widetilde{E}) := R(G, S)$ such that 
    \begin{enumerate}[(i)]
	    \item the reference measure of all spins in $ \widetilde{V}$ is the same
        \item the boundary correlation functions are unchanged:
		\begin{equation}\label{eq:goalcorr}
		    \langle \vp({A}) \rangle = \langle \widetilde{\vp}(A) \rangle_{ \widetilde{G}}
            \end{equation}
            for all $A \subset W$. 
    \end{enumerate}
\end{lemma}
\begin{figure}[H]
	\centering
	\includegraphics[scale=0.45]{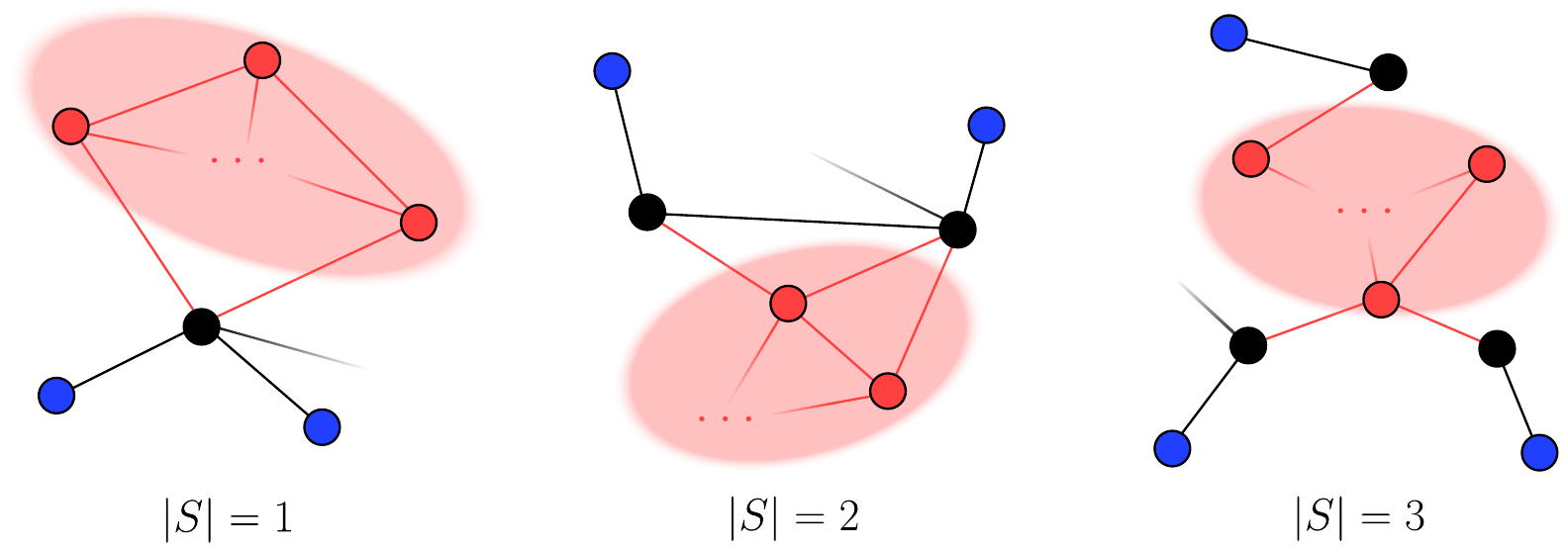}
	\caption{The three cases which can occur in Lemma \ref{lemma:firstmove}. Vertices of the sets $W$, $S$ and $X$ are shown in the colours blue, black and red respectively. In general the sets $W$ and $S$ may overlap even if it is not shown here. In the proof all the spins and interactions within $X$ are integrated away into new interactions between vertices of $S$.  
}
	\label{fig:reduction}
\end{figure}
\begin{proof}
	We write $X \subset V$ for the set of vertices which are separated from $W$ by $S$ so that $ \widetilde{V} = V \setminus X$. 
	We define a map $F : \left\{ \pm 1 \right\}^S \to \R$ given by
	\begin{equation}\label{eq:Fmap}
		F(\sigma) = \int_{\R^X} \exp\Big( \sum_{uu' \in E(S)} J_{uu'} \sigma_u \sigma_{u'} + \sum_{\substack{u \in S \\ x \in X}} J_{ux} \sigma_u \vp_x +\sum_{xx' \in E(X)} J_{xx'} \vp_x \vp_{x'}\Big) \, \d\l^X(\vp_X).
	\end{equation}
	Applying Lemma \ref{lemma:isingreplacement} to the spin system on $S \cup X$ with $H = S$, there must exists $ K_{uu'} \geq 0$ for each $u,u' \in S$ and $C > 0$ such that
	\begin{equation}\label{eq:beforelincomb}
		\sum_{\sigma \in \left\{ \pm 1 \right\}^S} \sigma(A) F(\sigma) = C \times \sum_{\sigma \in \left\{ \pm 1 \right\}^S} \sigma(A) \exp\Big(\sum_{u,u' \in S} K_{uu'} \sigma_u \sigma_{u'}\Big)  
	\end{equation} 
	for each $A \subset S$ (including the empty set). 
	Note that we have taken $r_u = 1$ for each $u \in S$ since $S \subset \Vis$. 
	This implies, by taking suitable linear combinations of \eqref{eq:beforelincomb}, that
	\begin{equation}\label{eq:afterlincomb}
		F(\sigma) = C \times \exp\Big(\sum_{u,u' \in S} K_{uu'} \sigma_u \sigma_{u'}\Big)
	\end{equation} 
	for each $\sigma \in \left\{ \pm 1 \right\}^S$.
	\noindent
	We can now construct the spin model on the reduced graph $ \widetilde{G} = ( V \setminus X, \widetilde{E})$. We set $ \widetilde{J_{uu'}} = K_{uu'}$ for each pair $u,u' \in S$ and set $ \widetilde{J}_{e} = J_{e} $ for every other edge in $\widetilde{E}$. Let $\widetilde{\vp}$ be distributed according to the spin model on $ \widetilde{G}$ with coupling constants $ \widetilde{J}$ and potentials $( P_v)_{v \in \widetilde{V}}$. To prove \eqref{eq:goalcorr} we have for all $A \subset W$ (by integrating over the spins $\vp_X$ and using $W \subset V \setminus X$)
	\[
		\left< \vp(A) \right> \propto \int_{\R^{ V \setminus X}} \vp\left( A \right) \exp\Big(\sum_{uv \in E \setminus E(S \cup X)} J_{uv} \vp_u \vp_v\Big) F((\vp_u)_{u \in S})) \d \l^{V \setminus X}(\vp_{V \setminus X})
\]
which by \eqref{eq:afterlincomb} (noting that $\vp_u \in \left\{ \pm 1 \right\}$ for $u \in S$) is proportional to
	\[
		\int_{\R^{V \setminus X}} \vp\left( A \right) \exp\Big(\sum_{uv \in E \setminus E(S \cup X)} J_{uv} \vp_u \vp_v\Big) \exp\Big(  \sum_{u,u' \in S} K_{uu'} \vp_u \vp_{u'}\Big)  \d \l^{V \setminus X}(\vp_{V \setminus X}),
\]
which is exactly (up to a scale factor, which matches by taking $A = \emptyset$) the correlation of $ \widetilde{\vp}(A)$ under the new measure. 
\end{proof}

The argument for moves of the second type is similar but with the additional difficulty that we must not only change the potential on vertices of $S \setminus \Vis$ but we must do so whilst preserving the boundary correlation functions, even if $S \setminus \Vis$ contains vertices of $W$.
\begin{lemma}\label{lemma:secondmove}
	Let $S \subset V$ be such that $|S| \leq 3$. Suppose any vertex $v \in S \setminus \Vis$ is only adjacent to vertices which are separated from $W$ by $S$. 
	Then, there is a classical spin model $ \widetilde{\vp}$ on the reduced graph $ \widetilde{G} = ( \widetilde{V}, \widetilde{E}) := R(G,S)$ and scaling factors $r_u > 0$ for each $u \in S$ such that
	\begin{enumerate}[(i)]
		\item the reference measure on all spins in $S \setminus \Vis$ becomes the Ising measure $\l\propto \delta_{1} + \delta_{-1}$,
			and remains the same on all other spins in $ \widetilde{V}$.
		\item  the boundary correlation functions satisfy
		\begin{equation}\label{eq:goalcorr2}
			\langle \vp({A}) \rangle = \Big( \prod_{u \in A}  r_u \Big) \langle \widetilde{\vp}(A) \rangle_{ \widetilde{G}}
            \end{equation}
            for all $A \subset W$. 
	\end{enumerate}
\end{lemma}
\begin{proof}	
	The argument is similar to the previous lemma except that we must be careful that changing the potential on $S \setminus \Vis \subset W$ does not affect the boundary correlations. 
	Let $X \subset V$ be the set of vertices in $G \setminus S$ which are separated from $W$ by $S$. We introduce the notation $\Sis = S \cap \Vis$ and set
	\[
	X' = X \cup (S \setminus \Sis).
	\] 
	For each subset $B \subset S \setminus \Sis$, we define a map $F_B : \left\{ \pm 1 \right\}^{\Sis} \to \R$ given by
	\begin{equation}\label{eq:Fmap2}
		F_B(\sigma) = \int_{\R^{X'}} \vp(B) \exp\Big(\sum_{uu' \in E(\Sis)} J_{uu'} \sigma_u \sigma_{u'} + \sum_{\substack{u \in \Sis \\ x \in X'}} J_{ux} \sigma_u \vp_x +\sum_{xx' \in E(X')} J_{xx'} \vp_x \vp_{x'}\Big) \, \d\l^{X'}(\vp_X).
	\end{equation}
	We now apply Lemma \ref{lemma:isingreplacement} to the spin system on $\Sis \cup X' = S \cup X$ with $H = S$, leading to couplings $ K_{uu'} \geq 0$ for each $u,u' \in S$, scale factors $r_u > 0$ for each $u \in S \setminus \Sis$ and $C > 0$ such that
	\begin{equation}\label{eq:beforelincomb2}
		\begin{split}
		&\sum_{\sigma \in \left\{ \pm 1 \right\}^{\Sis}} \sigma(A \cap \Sis) F_{A \setminus \Sis}(\sigma) \\
		&\hspace{1cm}= C  \Big( \prod_{u \in A \setminus \Sis} r_u \Big)  \sum_{\sigma \in \left\{ \pm 1 \right\}^S} \sigma(A) \exp\Big (\sum_{u,u' \in S} K_{uu'} \sigma_u \sigma_{u'}\Big)  \\
		&\hspace{1cm}= C \Big( \prod_{u \in A \setminus \Sis} r_u \Big) \\
		&\hspace{1.75cm}\times\sum_{\sigma \in \left\{ \pm 1 \right\}^{\Sis}} \sigma(A \cap \Sis) \sum_{\widetilde{\sigma} \in \left\{ \pm 1 \right\}^S} \widetilde{\sigma}(A \setminus \Sis) \id(\widetilde{\sigma}\mid_{\Sis} = \sigma) \exp\Big (\sum_{u,u' \in S} K_{uu'} \widetilde{\sigma}_u \widetilde{\sigma}_{u'}\Big)
		\end{split}
	\end{equation} 
	for each $A \subset S$ (including the empty set). Here we can only (and do) take $r_u = 1$ for each $u \in \Sis$. 
	Taking again linear combinations of \eqref{eq:beforelincomb2}, we have that 
	\begin{equation}\label{eq:afterlincomb2}
		\begin{split}
		&\sigma(A \cap \Sis) F_{A \setminus \Sis}(\sigma) \\
		&\hspace{1cm}= C   \Big( \prod_{u \in A \setminus \Sis} r_u \Big) \sum_{ \substack{ \widetilde{\sigma} \in \left\{ \pm 1 \right\}^S }} \i(\widetilde{\sigma}|_{\Sis} = \sigma) \times \widetilde{\sigma}(A)\exp\Big(\sum_{u,u' \in S} K_{uu'} \widetilde{\sigma}_u \widetilde{\sigma}_{u'}\Big)
		\end{split}
	\end{equation} 
	for each $\sigma \in \left\{ \pm 1 \right\}^{\Sis}$ and $A \subset S$.
	It should be clear how to use \eqref{eq:afterlincomb2} to finish the proof by arguing exactly as in Lemma \ref{lemma:firstmove}, the only observation to make is that for each $A \subset W$
	\begin{align*}
		\left< \vp(A) \right>_{G} \propto \Big(& \prod_{u \in A \cap  (S \setminus \Sis)} r_u \Big) \int_{\R^{V \setminus X'}} \vp(A \setminus X') \\&
		\times \exp\Big( \sum_{uv \in E \setminus E(S \cup X)} J_{uv} \vp_u \vp_v\Big ) F_{A \cap X'}((\vp_u)_{u \in \Sis})  \d\l^{V \setminus X'}(\vp_{V \setminus X'}),
	\end{align*}
	since crucially, under the assumption on $S$, any edge adjacent to a vertex of $S \setminus \Sis$ must belong to $E(S \cup X)$. 
\end{proof}

\begin{proof}[Proof of Lemma \ref{lemma:isingreplacement}]
	We demonstrate only the case $|H| = 3$ so let $H = \left\{ u_1,u_2,u_3 \right\}$. Let 
	\[
	x_A = \frac{1}{Z}\sum_{ \sigma \in \left\{ \pm 1 \right\}^H} \sigma(A) \exp\Big( \sum_{u,u' \in H} K_{uu'} \sigma_u \sigma_{u'} \Big) \quad \text{ for any $A \subset H$}
	\]
	with $Z$ chosen so that $x_{\emptyset} = 1$. Then, it may be easily checked that the only conditions for a triple $(x_{u_1u_2},x_{u_2u_3},x_{u_3u_1}) \in \left[ 0,\infty \right)^3$ to arise in this manner are the inequalities 
	\begin{equation}\label{eq:isinggriffeq}
		x_{u_1u_3} x_{u_2u_3}\leq x_{u_1 u_2} \quad \text{ and } \quad x_{u_1 u_2} \leq 1
	\end{equation} 
	and their symmetrisations. 
	That is, we must first show (for any ordering of $H$)
	\begin{equation}\label{eq:griffineq}
		\left( \frac{1}{r_{u_3}} \right)^2 \left< \vp_{u_1}\vp_{u_3} \right>_{G} \left< \vp_{u_2} \vp_{u_3} \right>_{G} \leq \left< \vp_{u_1} \vp_{u_2} \right>_{G}.
	\end{equation} 
	Ignoring the easy situation when $\left< \vp_{u_1} \vp_{u_2} \right>_{G}$ vanishes, we see that we may always choose $r_{u_3}$ sufficiently large so that \eqref{eq:griffineq} holds. If it so happens that $u_3 \in \Vis$, then
	\[
		\left< \vp_{u_1} \vp_{u_2} \right>_{G} = \left< \vp_{u_1} \vp_{u_2} \vp_{u_3}^2 \right>_{G} \geq \left< \vp_{u_1}\vp_{u_3} \right>_{G} \left< \vp_{u_2} \vp_{u_3} \right>_{G},
	\] 
	where we have used $\vp_{u_3}^2 = 1$ and the Griffiths inequality for $\left< \cdot \right>_{G}$, which is exactly \eqref{eq:griffineq} with $r_{u_3} = 1$.
	Finally, the second requirement in \eqref{eq:isinggriffeq} certainly holds so long as $r_{u_i} \geq ( \left< \vp_{u_i}^2 \right>_{G})^{1 / 2}$, which again permits $r_{u_i} = 1$ if $u_i \in \Vis$.
\end{proof}

\section{Proof of Theorem \ref{thm:pfaffian}}\label{sec:proof1}

\subsection{Graph-theoretic preliminaries. }We say that a finite graph $G = (V, E)$ with a distinguished set of vertices $W \subset V$ can be 
\emph{drawn in the disk}
if there exists an embedding of $G$ into $\mathbb{D} = \{x: |x| \leq 1\} \subset \RR^2$ with $W$ drawn on $\partial \mathbb{D}$, 
such that there are no crossings of the edges. 
In particular, if $G$ can be drawn in the disk, then it is planar and all vertices of $W$ lie on a single face. 

\begin{proposition} \label{prop:main}
    Let $G = (V, E)$ be a finite graph with a distinguished tuple $W$ of vertices. 
    If $G$ satisfies the $W$-crossing property and is fully reduced, then $G$ can be drawn in the disk. 
\end{proposition}
We will use this only once at the end of the proof to deduce that our constructed graph is indeed planar. A proof of Proposition \ref{prop:main} is included in Appendix \ref{sec:appendix}. 
It was realised by the authors after writing that a different proof already appeared in the work of 
Robertson and Seymour~\cite{RobSey},
as part of their proof of Wagner's well-quasi-ordering conjecture. 

The other result we will use is a straightforward application of Menger's theorem \cite{Menger} (see Lemma \ref{lemma:Mengerstheorem} below). 
We say that two paths $\tau_1$ and $\tau_2$ in $G$ are $U$\emph{-disjoint} if 
\[
\tau_1 \cap \tau_2 \cap U = \emptyset.
\] 
\begin{lemma}\label{lemma:newmenger}
	Let $A,B \subset V$ be disjoint and also let $U \subset V$ be disjoint from $A$. Let $k$ be the minimum size of a set $K \subset U$ separating $A$ from $B$.
	Then, there exists a collection of pairwise $U$-disjoint paths $\path_1,\ldots,\path_k$ from $A$ to $B$. 	
	If there is no such set, then there is a path $\tau$ from $A$ to $B$ avoiding $U$. 
\end{lemma}
The lemma does not imply that the endpoints in $A$ or $B$ are necessarily distinct. Only vertices in $A \cap U$ or $B \cap U$ necessarily appear in at most one path. 
\begin{remark}\label{remark:endpoint}
	It is easy to see that whenever both $k \geq 2$ and $|B| \geq 2$, we can choose our paths carefully so that so that not every one has the same endpoint in $B$. The same statement holds for $A$ by symmetry.
\end{remark}
\begin{proof}
	We define $N_v$, for any $v \not\in U$, to be the set of vertices in $U$ which are neighbour to a vertex in the connected component of $v$ in $G \setminus U$, and otherwise set $N_v = \left\{ v \right\}$ if $v \in U$.
	We apply Lemma \ref{lemma:Mengerstheorem} to the path-induced graph $G'$ from $G$ by $U$ with the subsets
	\[
		A' = \bigcup_{v \in A} N_v \quad \text{ and }\quad B' = \bigcup_{v \in B} N_v.
	\]
	Note that any separating set $K \subset U$ in $G$ of $A$ and $B$ is also a separating set in $G'$ of $A'$ and $B'$. 
	By lifting back to $G$, we thus obtain a family of $U$-disjoint paths $\tau_1,\ldots,\tau_k$ from $A$ to $B$.
	\end{proof}
\subsection{The algorithm. }
We assume that $\left< \cdot \right>_{G}$ is a spin model satisfying Pfaffian relations on $W$. In particular, the graph $G$ satisfies the $(W,\Vis)$-crossing property (Corollary \ref{cor:pfaff}).
\\[1em]
\noindent Our approach is based on an algorithmic reduction of the graph $G$. 
The algorithm is initialised with the pair
\[
	(G^{(0)},\Vis^{(0)}),
\]
where $G^{(0)} = G$ and $\Vis^{(0)} = \Vis$ is the set of Ising-type vertices of the spin model $\left< \cdot \right>_{G}$. 
It will produce a sequence
\[
	(G^{(0)},\Vis^{(0)}),(G^{(1)},\Vis^{(0)}), \ldots, (G^{(T)},\Vis^{(T)}),
\]
with $\k{G} = (\k{V},\k{E})$ a graph and $\Vis^{(k)} \subset \k{V}$ a subset of the vertices, ending in a finite number $T$ of steps. Before describing the algorithm, it will be helpful to state properties that each graph $G^{(k)}$, $k \geq 1$, will have:
\begin{enumerate}[(1)]
	\item The subset $W$ (and its ordering) is always preserved (i.e.~$W \subset V^{(k)}$). 
	\item $\k{G}$ is obtained from the previous graph $G^{(k-1)}$ by applying a reduction move
		\[
			G^{(k-1)} \longrightarrow \k{G} := R(G^{(k-1)},S^{(k-1)})
		\] 
		for a suitable set $S^{(k-1)}$. 
		In particular, $\k{V} \subset V^{(k-1)}$. 
		We also always set 
		\[
			\Vis^{(k-1)} \longrightarrow \Vis^{(k)} := (\Vis^{(k-1)} \cap \k{V}) \cup S^{(k-1)}
		\] 
	\item Each graph $G^{(k)}$ satisfies the $(\k{G},\Vis^{(k)})$-crossing property.
\end{enumerate}
To define the algorithm, our main task is to choose a suitable reduction set $S^{(k)}$ at each step. This choice will be taken according to one of three stages that the algorithm is in:
\begin{itemize}
	\item In the first stage of the algorithm $0 \leq k \leq T_1$, we only apply the reduction moves of the second kind. We use Lemma \ref{lemma:secondmove} to convert vertices of $W \cap \Vis^{(k)}$ into Ising-type vertices. After this stage, the resulting graph $G^{(T_1)}$ will have the property that $W \subset \Vis^{(T_1)}$ which will be of crucial importance later on. 

	\item In the second stage $T_1 \leq k \leq T_2$, we only apply the reduction moves of the first kind. We use Lemma \ref{lemma:firstmove} using reduction sets $S \subset \Vis^{(k)}$ until all remaining non-Ising type vertices have been removed from $V$. This leaves us with the graph $G^{(T_2)}$ which has $\Vis^{(T_2)} = V^{(T_2)}$.

	\item  In the third and final stage $T_2 \leq k \leq T_3$, we apply Lemma \ref{lemma:firstmove} again with any remaining possible reduction sets $S \subset \k{V}$ (without now having to worry about the condition $S \subset \Vis^{(k)}$). We then prove that the graph $G^{(T_3)}$ obtained at the end of all this is \emph{planar}.
\end{itemize}
\begin{figure}\centering
\hspace{-0.3cm}	\includegraphics[scale=0.8]{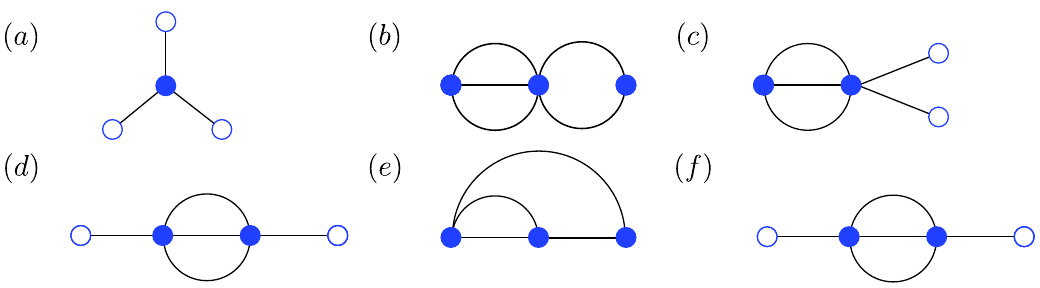}
	\caption{Configurations appearing in the proof of Theorem \ref{thm:pfaffian}. All vertices belong to $W$ and the solid blue vertices belong to $V \setminus \Vis$. Drawn edges correspond to paths between the two endpoints and all the paths in a given diagram are $\Vis$-disjoint. Each diagram with four vertices violates the $(W,\Vis)$-crossing condition for any ordering on $W$. In any diagram with three vertices, there can be no further vertex of $W$ in same connected component without again violating the condition. }
	\label{fig:forbidden}
\end{figure}
\begin{proof}[Proof of Theorem \ref{thm:pfaffian}]
	We will use the algorithm outlined above and along the way prove that it does indeed terminate after a finite number of steps in a planar graph. Suppose that we have reached step $k \geq 0$, the pair $(\k{G},\Vis^{(k)})$ satisfies the properties above and we are in one of the three following stages. 
	\paragraph{Stage I: Removal of vertices in $v \in W \setminus \Vis$.}
	Assuming that there is still a vertex $v \in W \setminus \Vis^{(k)}$, we would like to apply a move of the second type for some set $S \subset V$ which contains $v$. Using Lemma \ref{lemma:secondmove}, this will ensure that $v \in \Vis^{(k+1)}$ and so this process will terminate once there are no such vertices left. 
	It remains to argue that there is such a set $S$ (which we will actually take to be a subset of $\Vis^{(k)} \cup W$) satisfying the desired assumptions. This is the subject of the next claim.
	\\[1em]
	\noindent\textbf{Claim: }Suppose $G$ satisfies the $(W,\Vis)$-crossing property and $w \in W \setminus \Vis$. Then, there exists $S \subset \Vis \cup W$ with $w \in S$ such that (1) $|S| \leq 3$ and (2) any vertex $v$ in $S \setminus \Vis$ is only adjacent to vertices separated from $W$ by $S$.
	\\[1em]
	The argument involves several cases; in each of which we either construct a valid set $S$ 
	(i.e.~satisfying the required conditions) 
	or show that it is inconsistent with the crossing condition (for any ordering of $W$) by reducing to one of the diagrams in Figure \ref{fig:forbidden}.
	\\[1em]
	\noindent To start, we assume that the intersection of $W$ with the connected component of $w$ is of size at least four since otherwise it is automatically a valid set.
	We apply Lemma \ref{lemma:newmenger} in $G$ with $A = \left\{ w \right\}$, $B = W \setminus \left\{ w \right\}$ and $U = \Vis \cup \left( W \setminus \left\{ w \right\}  \right)$. 
	Let $K$ be a minimal separating set which must exist since $B \subset U$.
	\\[1em]
	\noindent\emph{Case 1: $|K| \geq 3$.} We show this cannot happen. Using Lemma~\ref{lemma:newmenger}, we see there are paths $\tau_1,\tau_2$ and $\tau_3$ from $w$ to $W \setminus \left\{ w \right\}$ which are $\Vis \cup (W \setminus \left\{ w \right\})$-disjoint and thus all have distinct endpoints in $W \setminus \left\{ w \right\}$. This is depicted in Figure \ref{fig:forbidden} (a) and clearly violates the crossing condition.
	\\[1em]
	\noindent\emph{Case 2: $K \subset \Vis$ and $|K| \leq 2$.} Then the set $S = \left\{ w \right\} \cup K$ is a valid set by construction.
	\\[1em]
	\noindent\emph{Case 3: $K \setminus \Vis$ is non-empty and $|K| \leq 2$.} Let $X$ be the set of vertices separated from $W \setminus \left\{ w \right\}$ by $K$. We will work with the two subgraphs of $G$ given by $G_1 = X$ and $G_2 = (V \setminus X) \cup K$, both of which include $K$, which we refer to as `inside' and `outside' $K$ respectively. We denote the (at most two) elements of $K$ by $w' \in W \setminus \Vis$ and $v$ (with $v = w'$ if $|K| = 1$).
	\\[1em]
	\noindent We start by working inside $K$. We apply Lemma \ref{lemma:newmenger} in the graph $G_1$ with $A = \left\{ w \right\} $, $B = K$ and $U = \Vis \cap G_1$; let $K_1$ be a minimal separating set.
	\begin{itemize}
		\item If $K_1$ does exist and $|K_1| \leq 2$, then the set $S = \left\{ w \right\} \cup K'$ is a valid set. 
		\item If either $K'$ does not exist or if $|K_1| \geq 3$, then there must be three $\Vis$-disjoint paths $\tau_1,\tau_2,\tau_3$ in $G_1$ with $\tau_1,\tau_2$ both from $w$ to $w' \in K \setminus \Vis \subset W$. We can also ensure the third path ends at $v$, see Remark \ref{remark:endpoint}. We will use this construction later.
	\end{itemize}
	At this point, we have to split our argument in two according to whether $|K \setminus \Vis| = 1$ (which includes $|K| = 1$) or not. 
	\\[1em]
	\noindent\emph{Case 3.1: }Let us first assume the former.
	We must look outside $K$ for a separating set. We apply Lemma \ref{lemma:newmenger} in the graph $G_2$ with $A = K$, $B = W \setminus (K \cup \left\{ w \right\})$ and $U = \Vis$; let $K_2$ be a minimal separating set (which may contain $v$ if $v \in \Vis$).
	\begin{itemize}
		\item Suppose that $K_2$ does exist and $|K_1| \leq 1$, then the set $S = \left\{ w,w' \right\} \cup K_2$ is a valid set. 
		\item We argue that if either $K_2$ does not exist or if $|K_2| \geq 2$ then there is a contradiction. Assuming so, there must be $\Vis$-disjoint paths $\tau_4,\tau_5$ in $G_2$ from $K$ to $W \setminus (K \cup \left\{ w \right\})$ with $\tau_4$ starting at $w'$ and $\tau_5$ starting at $v$ (again see Remark \ref{remark:endpoint}).
			If $|K| = 1$, the two possibilities of the five paths $\tau_1,\ldots,\tau_5$ are shown in Figure \ref{fig:forbidden} (b) and (c).
		If instead $|K| = 2$, then the paths $\tau_1,\tau_2,\tau_4,\tau_3 \circ \tau_5$ must appear as in Figure \ref{fig:forbidden} (d) or (e). All four of these diagrams are incompatible with the crossing condition.
	\end{itemize}
	
	\noindent\emph{Case 3.2: }Finally, we turn to when $K \subset V \setminus \Vis$ and $|K| = 2$, so that now both $w',v \in W \setminus \Vis$. 
	We return to the geometry inside $K$. We again invoke Lemma \ref{lemma:newmenger} in $G_1$ with $A = \left\{ w,w' \right\}$, $B = \left\{ t \right\}$ and $U = \Vis$. Like before, if the minimal separating set $K_3$ is a singleton then the set
	\[
	S = \left\{ w,w' \right\} \cup K_3
	\] 
	will be valid. We now rule out either $K_3$ not existing or $|K_3| \geq 2$. As will now be familiar, there are $\Vis$-disjoint paths $\tau_6$ and $\tau_7$ in $G_1$ both starting at $w$ and ending at $w'$ and $v$ respectively.  	
	If $W \setminus (K \cup \left\{ w \right\})$ is disconnected from $K$ then the set $S = \left\{ w \right\} \cup K$ suffices. Otherwise, there is a path $\tau_8$ in $G_2$ starting from $w'$ (say) to $ W \setminus \left\{ w,w',v \right\}$. The paths $\tau_6,\ldots,\tau_8$ are shown in Figure \ref{fig:forbidden} (f), and again lead to a contradiction.
	\paragraph{Stage II: Removal of remaining vertices in $V \setminus \Vis$. }
	At this point, we may assume that $W \subset \Vis^{(k)}$. Recall that our strategy is to apply Lemma \ref{lemma:firstmove} until all the non-Ising type vertices have been removed. The following claim asserts that unless $V^{(k)}\setminus\Vis^{(k)}$ is empty, in which case this stage of the algorithm is considered completed, there exists a reduction set $S$ which satisfies the key assumption ($S \subset \Vis^{(k)}$) in Lemma \ref{lemma:firstmove} and such that in the reduced graph
	\[
		| V^{(k+1)} \setminus \Vis^{(k+1)} | \leq | V^{(k)} \setminus \Vis^{(k)} | - 1.
	\]
	Thus, after finitely many steps the second stage will have ended.
	\\[1em]
	\noindent\textbf{Claim: }Suppose $G$ satisfies the $(W,\Vis)$-crossing property with $W \subset \Vis$ and let $v \in V \setminus \Vis$. Then, there exists a reduction set $S$ contained in $\Vis$ which separates $\left\{ v \right\}$ from $W$.
	\\[1em]
	\noindent We argue similarly to above. Assume that the claim is false and so the minimal such set $S \subset \Vis$ is of size at least $4$. 
	Using Lemma \ref{lemma:newmenger} with $A = \left\{ v \right\}$ and $B  =\left\{ W  \right\}$, we see that there must exist paths
	\[
		\path_1,\ldots,\path_4,
	\]
	with $\path_i$ connecting $v$ to $w_i$, and such that $\tau_i \cap \tau_j \cap \Vis = \emptyset$ for $1 \leq i < j \leq 4$. Each $w_i$ must, since $W \subset \Vis$, be distinct. This is illustrated in Figure \ref{fig:forbidden} (b), and clearly leads to a contradiction of the $(W,\Vis)$-crossing property, proving the claim.
	\paragraph{Stage III: Reducing to a planar graph.}
	From now on, we simply choose any reduction set $S^{(k)} \subset \Vis^{(k)}$ and apply Lemma \ref{lemma:firstmove} (using $\Vis^{(k)} \subset \k{V}$). This process will terminate after a finite number of steps since in the reduced graph 
	\[
		|V^{(k+1)}| \leq |V^{(k)}| -  1
	\] 
	and there can be no such reduction set if $V^{(k)} = W$ (though it will usually stop before this point). The final graph $G^{(T_3)}$ has the $W$-crossing property and is fully-reduced, and so by Proposition \ref{prop:main} can be drawn in the disk. 
	\\[1em]
	\noindent\emph{Conclusion.} It only remains to clarify why this proves the two statements in Theorem \ref{thm:pfaffian}. Let 
	\[
		U = S^{(0)} \cup S^{(1)} \cup \ldots \cup S^{(T_3 - 1)} \subset \Vis^{(0)} \cup W.
	\] 
	Since we only used reduction moves at each step, it must be that the graph $G^{(T_3)}$ is actually the path-induced graph from $G$ by the subset $U$. Moreover, it is the content of Lemma \ref{lemma:secondmove} and Lemma \ref{lemma:firstmove} that for each $k \geq 1$, there is a spin model $ \widetilde{\vp}^{(k)} \sim \left< \cdot \right>_{\k{G}}$ on $\k{G}$ such that
\begin{itemize}
	\item the set $\Vis^{(k)}$ is exactly the set of Ising-type vertices for $\left< \cdot \right>_{\k{G}}$
	\item the boundary correlation functions of $ \widetilde{\vp}^{(k)}$ coincide (up to a rescaling) with those of $\left< \cdot \right>_{G}$. 
\end{itemize}
	In particular, there exists $r_v > 0$ for each $v \in W$ such that
	\[
		\left< \vp(A) \right>_{G} =\Big ( \prod_{v \in A} r_v\Big) \left< \widetilde{\vp}^{(k)}(A) \right>_{G^{(T_3)}} \quad \text{ for each $A \subset W$}.
	\] 
	By absorbing the scale factors $r_u$ into the Ising-type potential for each $u \in V^{(T_3)}$, we see that we have completed our proof of the theorem.
\end{proof}

\begin{remark}
	The algorithm above involves a frequent number of choices at each step however they turn out to not matter: the final product can be shown to be \emph{unique} up to graph isomorphisms. 
	We phrased the algorithm in this staged manner since otherwise it is possible for it to become stuck: there are no available reduction sets $S \subset \Vis^{(k)}$ to remove a vertex $v \in \k{V} \setminus \Vis^{(k)}$ until first a move of the second type is applied.
\end{remark}
\part{Sharpness}\label{part:2}

\section{Correlation inequalities}\label{sec:corrineq}
We resume the setting of Section \ref{sec:app2}. In particular, we fix $P \in \Ec$ and consider the finite-volume spin models $\left< \cdot \right>_{G_n,\beta,h}$ which (under Assumption \ref{asum:two}) converge as $G_n \to \G$ and then $h \downarrow 0$ to the infinite-volume measure $\left< \cdot \right>_{\beta}^{+}$. 
\\[1em]
\noindent This section concerns the Simon-Lieb inequality and the FKG inequality for $\P^{+ / +}_\beta$, the two essential ingredients in the proof of Theorem \ref{thm:sharpness}. 
We will frequently take for granted various consequence of the Griffiths inequalities \cite{GriffithsCor,Gin}, in particular, that the map 
\[
\beta \mapsto \left< \vp_o \right>_{G,\beta,h}
\]
is non-decreasing in $\beta \geq 0$. 
The interested reader is sent to \cite[Chapter 2]{simon2026phase} for an exhaustive discussion on the variety of correlation inequalities available.
\subsection{Simon-Lieb inequality. }
The version of the inequality that we will use is as follows.
We identify a subset of the vertices $S \subset V$ with the subgraph $(S,E(S))$ of $G$.

\begin{proposition}[] Let $G = (V,E)$ be a finite graph and $S \subset V$. If $P \in \Ec$, then
	\[
		\left< \vp_o \vp_x \right>_{G,\beta,h} \leq \beta \sum_{\substack{u \in S \\ v \in G \setminus S}} J_{uv} \left< \vp_o \vp_u \right>_{S,\beta,0} \left< \vp_v \vp_x \right>_{G,\beta,h} + h \sum_{u \in S} \left< \vp_o \vp_u \right>_{S,\beta,0} \left< \vp_x \right>_{G,\beta,h} 
	\] 
	for any $o \in S$ and $x \in V \setminus S$.
\end{proposition}

This result was proved in \cite[Theorem 6.1]{BFS}, under a slightly different assumption on the potential. That this assumption is implied by $P \in \Ec$ is stated in \cite[Section 4]{brydges1983random}{\hypersetup{linkcolor=black}\footnote{Note that this implication is falsely claimed to not hold in \cite{BFS} but this is later corrected in \cite{brydges1983random}}}. 
We note that the argument in \cite{BFS} relied on a Gaussian domination inequality, for which a weaker version had previously been established for the class $\Ec$ \cite{newman,simon2026phase}. The adaptation to a non-negative magnetic field (as stated here) is immediate through the use of a ghost vertex.
\\[1em]
\noindent For any finite subset $S \subset \G$ containing $o \in V$, we define
\[
	\vp_\beta\left( S\right) = \beta \sum_{\substack{u \in S \\ v \not\in S}} J_{uv} \left< \vp_o \vp_u \right>_{S,\beta,0}.
\] 
This definition first appeared for the Ising model in \cite{DCT}.
Note that by our assumption on $J$~\eqref{eq:Jasum}, $\vp_\beta(S) < \infty$ for any finite set $S$.
The following is a classical consequence of the Simon-Lieb inequality, and it may be proved by arguing exactly as in \cite{DCT}.
\begin{lemma}\label{lemma:expdecay}
	Suppose that there exists a finite set $S \subset \G$ such that $\vp_\beta(S) < 1$. Then,
	\[
	\chi(\beta) < \infty
	\] 
	and if $J$ is finite-range, then $\left< \vp_o \vp_x \right>^+_\beta$ decays exponentially fast in $d(0,x)$.
\end{lemma}

\subsection{FKG and stochastic domination. }We include a brief review of certain aspects of the theory in the setting that we will need them. Let $T_1,\ldots,T_n$ be totally ordered sets (we will often take either $[0,\infty)$ or $\left\{ 0,1 \right\}$ with the natural order or its reverse) and so that $\Omega = T_1 \times \ldots \times T_n$ is a distributive lattice. A measure $\mu$ on $\Omega$ satisfies the \emph{FKG inequality} if  
\[
	\mu\left[ FG \right] \geq \mu\left[ F \right] \mu\left[ G \right] 
\] 
for any two increasing measurable functions $F,G : \Omega \to \left[ 0,\infty\right)$. If $X \sim \mu$ then we may instead say $X$ satisfies the FKG inequality. If $(X,Y)$ are two coupled random variables on $\Omega$, we say that $(X,Y)$ (resp.~$(X,-Y)$) satisfies the FKG inequality if there joint measure does on $\Omega \times \Omega$ with the product order (resp.~with the product order but reversed on the second component).
An important example is that if $\mu_1,\ldots,\mu_n$ are probability measures on the totally ordered sets $T_1, \ldots, T_n$, then the product measure $\mu = \mu_1 \otimes \ldots \otimes \mu_n$ satisfies the FKG inequality \cite{harris1960lower}. If $\nu$ is another measure on $\Omega$ absolutely continuous with respect to a product measure $\mu$ such that the relative density
	\[
	\H = \frac{\d \nu}{\d \mu}
	\] 
	is strictly positive, then the \emph{FKG-lattice condition} is that
	\[
		\H(x \vee x')  \H(x \wedge x') \geq \H(x)  \H(x') \quad \text{ for any } x,x' \in \Omega.
	\] 
This condition implies that $\nu$ also satisfies the FKG inequality \cite{FKG}{\hypersetup{linkcolor=black}\footnote{The version for continuous spins may be found in e.g.~\cite{GRS,preston1974generalization}}}. We say that one measure $\nu_1$ stochastically dominates another measure $\nu_2$ (denoted $\nu_1 \succeq \nu_2$) if 
\[
\nu_1\left[ F \right] \geq \nu_2\left[ F \right] 
\] 
for each increasing measurable function $F : \Omega \to \left[ 0,\infty \right)$. If $\nu_1$ and $\nu_2$ are two laws of a coupled pair $(X,Y)$ as above, we may write
\[
	\nu_1 \overset{(X,Y)}{\succeq} \nu_2 \quad \text{ or }\quad \nu_1 \overset{(X,-Y)}{\succeq} \nu_2
\] 
to specify the precise order on the product space that we are using.
We will often use the following fact.
\begin{lemma}\label{lemma:stochdom}
	Let $\nu_1$ be absolutely continuous with respect to $\nu_2$. If $\nu_2$ satisfies the FKG inequality and the relative density $\d \nu_1 / \d \nu_2$ is increasing (resp.~decreasing) on $\Omega$, then $\nu_1 \succeq \nu_2$ (resp.~$\nu_1 \preceq \nu_2$)\footnote{Note that this does not imply $\nu_1$ satisfies the FKG inequality.}. 
\end{lemma}
\noindent To end this section we include a proof of the FKG inequality for the representation. As part of the proof, we will verify the FKG-lattice condition for the density of $(\L^+,\L^-)$ under $\P^{+ / +}_{\beta}$ so we will first need to explicitly write down said density. To do so, we will need some new notation.
\\[1em]
\noindent Given any $S \subset V$, $\L \in [0,\infty)^S$ and $\w \in \left\{ 0,1 \right\}^{E(S)}$, we let
\begin{equation}\label{eq:qdef}
	q^S_{\L}(\w) = \prod_{uv \in E(S)} (1 - e^{-4J_{uv}\L_u\L_v})^{\w_{uv}} ( e^{-4J_{uv} \L_u \L_v})^{1-\w_{uv}}.
\end{equation}
and 
\begin{equation}\label{eq:qbardef}
	q_\L^{ \overline{S}} = q_\L^{S}(\w) \times \prod_{u \in S} \left( 1 - \exp(-4 h \L_u) \right)^{\w_{ug}} \left( \exp( - 4 h \L_u \right)^{1 - \w_{ug}}
\end{equation}
We will use the notation $q^G_{\L}$ and $q^{ \overline{G}}_{\L}$ for the most common case when $S = V$.
Using this notation, the density of $(\vpp,\vpm,\w^+,\w^-)$ under $\P_{\beta}^{+ / +}$ is proportional to
\begin{equation}\label{eq:rotden25}
	\i(\mathcal{A}) \times q^{ \overline{G}}_{\L^+}(\w^+) q^G_{\L^-}(\w^-) 
	\times \exp\big( 2\beta (\L^+,J\L^+) + 2\beta(\L^-, J\L^-) + 2(h,\L^+)\big) \d\l^V(\vpp + \vpm)\, \d\l^V(\vpp - \vpm),
\end{equation} 
where $\mathcal{A}$ is the event that $\sgn(\vpp)$ and $\sgn(\vpm)$ are constant on the nontrivial clusters of $\w^+$ and $\w^-$ respectively and $\sgn(\vpp) \equiv +1$ on the $\w^+$-cluster of the ghost.  
Let us now define
\begin{equation}\label{eq:Zdef}
	Z^S(\L) = \sum_{\w \subset E(S)} 2^{k(\w)} q^S_{\L}(\w) \quad \text{and }\quad Z^{ \overline{S}}(\L) = \sum_{ \w \subset \overline{E}(S)} 2^{k(\w)} q^{ \overline{S}}_{\L}(\w),
\end{equation} 
where $k(\w)$ is the number of connected components of $\w$.
The density of $(\L^+,\L^-,\w^+,\w^-)$ in $\P_{\beta}^{+ / +}$, obtained from \eqref{eq:rotden25} by summing over $\sgn(\vpp)$ and $\sgn(\vpm)$, is proportional to
\begin{multline}\label{eq:rotden3}
	2^{k(\w^+)} q_{\L^+}^{ \overline{G}}(\w^+) \times 2^{k(\w^-)} q_{\L^-}^G(\w^-) \times \exp(2\beta(\L^+,J\L^+) + 2\beta(\L^-,J\L^-) + 2(h,\L^+)) \\ \times \exp\Big(-\textstyle\sum_{u \in V}\I(\L^+_u,\L^-_u) \Big) \d\l_2^G|_{[0,\infty)}(\L^+) \d\l_2^G|_{[0,\infty)}(\L^-).
\end{multline}
Here, $\l_2^G|_{[0,\infty)}$ is the probability measure on $\left[ 0,\infty \right)^V$ obtained from $\l_2^G$ by projecting $\vp^\pm \to |\vp^\pm|$.
\\[1em]
\noindent We also need the following simple lemma for when $P \in \Ec$ has the form \eqref{eq:Pdef}. The first item in the lemma is the reason why we chose to define $\I$ relative to the measure $\l_2 \otimes \l_2$ rather than $\l \otimes \l$.
We leave the straightforward proof to the reader.
\begin{lemma}\label{lemma:Iprops}
	The function $\I = \I(x,y)$ defined in \eqref{eq:Idef} is such that:
	\begin{enumerate}
		\item $\I$ is non-decreasing in $x$ and $y$ for $x,y \geq 0$.
	\end{enumerate}
	Supposing further that $P$ has full support.  
	\begin{enumerate}
		\setcounter{enumi}{1}
		\item $\d \I / \d x$ is non-decreasing in $y$ and $\d \I / \d y$ is non-decreasing in $x$ for $x,y \in (0,\infty)$.
	\end{enumerate}
\end{lemma}
We may now finally prove the FKG inequality for $\P_\beta^{+ / +}$, which we do so first under the assumption that $P \in \Ec$ has \emph{full support} i.e.~it is as in \eqref{eq:Pdef} for $A = \infty$, so as to avoid any issues with relative densities which are not strictly positive.
To extend the result for $A \in (0,\infty)$ or for the Ising reference measure $\d\l$ (and therefore the entirety of $\Ec$), it suffices to note that these measures may be suitably approximated by those of full support and the FKG inequality is preserved under weak convergence. 
\begin{proof}[Proof of Proposition \ref{prop:FKG}]
	Our plan will be to argue each of the three claims:
	\begin{enumerate}[(1)]
		\item The law $(\L^+,-\L^-)$ satisfies the FKG inequality under $\P_{\beta,h}$.
		\item Conditionally on $(\L^+,\L^-)$, the percolation $\w^\pm$ satisfies the FKG inequality under $\P_{\beta,h}$. 
		\item The law of $\w^+$ (resp.~$\w^-$) conditionally on $(\L^+,\L^-)$ is increasing (resp.~decreasing) in $(\L^+,-\L^-)$.
	\end{enumerate}
	Together with the crucial conditional independence of $\w^+$ and $\w^-$, these imply the FKG inequality for the entire quadruplet using the so-called \emph{chain rule} of the FKG inequality (see \cite[Proposition 3.10]{glazman2025delocalisation} for similar reasoning). 
	\\[1em]
	\noindent\textbf{Claim (1). }This is the most substantial part of the proof. We will verify the FKG lattice condition for the density of $(\L^+,\L^-)$ under $\P_{\beta,h}$, which implies the desired FKG inequality \cite{FKG}. By summing over $(\w^+,\w^-)$ in \eqref{eq:rotden3}, we find this density relative to $\l_2^G|_{[0,\infty)} \otimes \l_2^G|_{[0,\infty)}$ is proportional to
	\begin{multline}\label{eq:absden}
	\underbrace{Z^{ \overline{G}}(\L^+) \times Z^G(\L^-)}_{\text{Term 1}} \times  \underbrace{ \exp( 2\beta(\L^+,J\L^+) + 2\beta (\L^-,J\L^-)}_{\text{Term 2}} + (2h,\L^+)) \times  \underbrace{\exp( -\textstyle\sum_{u \in V}\I(\L^+_u,\L_u^-) )}_{\text{Term 3}},
	\end{multline} 
	which we denote by $\H(\L^+,\L^-)$. Since $P$ has full support, $\H(\L^+,\L^-)$ is strictly positive on $(0,\infty) \times (0,\infty)$. 
	The condition that we need to show is
	\[
		\H(\L^+ \vee \widetilde{\L^+}, \L^- \wedge \widetilde{\L^-}) \times \H(\L^+ \wedge \widetilde{\L^+},\L^- \vee \widetilde{\L^-}) \geq \H(\L^+,\L^-) \times \H( \widetilde{\L^+}, \widetilde{\L^-}) 
	\]
for every $\L^+,\L^-, \widetilde{\L^+}, \widetilde{\L^-} \in \left[ 0,\infty\right)^V$.
	We deal with each term in turn. 
	\\[1em]
	\noindent\emph{Term 1. }The map $\L \mapsto Z^G(\L)$ is known to satisfy the FKG-lattice condition: 
	\[
		Z^G(\L \vee \L') \times Z^G(\L \wedge \L') \geq Z^G(\L) \times Z^G(\L') \quad\text{ for any } \L,\L' \in \left[ 0,\infty\right)^V, 
	\] 
	and similarly for $\L \mapsto Z^{ \overline{G}}(\L)$ since $h \geq 0$, see e.g.~\cite[Lemma 6.1]{LamOtt}.
	\\[1em]
	\noindent\emph{Term 2. }This is the easiest term. Indeed, we may treat 
	\[
		\exp(2\beta(\L^+,J\L^+)) \quad \text{ and }\quad \exp(2\beta(\L^-,J\L^-))
	\] 
	separately since there is no interaction between them. As we have
	\[
		\frac{\partial^2 }{\partial \L_u \partial \L_v} 2\beta \left( \L, J \L \right)  = 2\beta J_{uv} \geq 0
	\] 
	for any two distinct vertices $u,v \in V$, the map $\L \mapsto \exp(2\beta(\L,J\L))$ is log-convex which is equivalent to the desired lattice condition.
	\\[1em]
	\noindent\emph{Term 3. }This final term is the only place where the interaction between $\L^+$ and $\L^-$ must be considered. Clearly, it suffices to show that
	\[
		\I(x \vee \widetilde{x},y \wedge \widetilde{y}) +  \I(x \wedge \widetilde{x},y \vee \widetilde{y}) - \I(x,y) - \I( \widetilde{x}, \widetilde{y}) \leq 0
	\] 
	for any $x,y, \widetilde{x}, \widetilde{y} \geq 0$. Without loss of generality we may assume that $x \geq \widetilde{x}$ and also that $y \geq \widetilde{y}$ (otherwise the left hand side vanishes). Then, we can rewrite this as an integral,
	\[
		\int_{ \widetilde{x}}^{ x}  \frac{\d }{\d \tau} \I(\tau, \widetilde{y} ) - \frac{\d}{\d \tau}\I(\tau, y) \d \tau
	\] 
	where the integrand is non-positive by Lemma \ref{lemma:Iprops}. 
	\\[1em]
	\noindent\textbf{Claim (2) and (3). }The conditional law of $\w^+$ (resp.~$\w^-$) given $(\L^+,\L^-)$ is the FK-Ising model on $\overline{G}$ (resp.~$G$). The edge parameters are determined by $q_{\L^+}^{ \overline{G}}$ and $q_{\L^-}^G$ (see \eqref{eq:qbardef} and \eqref{eq:qdef}) and are increasing in $\L^+$ and $\L^-$ respectively. 
	Therefore, the final two claims follow from the known FKG and stochastic monotonicity properties of the FK-Ising model \cite{Grimmett}.
\end{proof}
\section{Proof of Theorem \ref{thm:sharpness}}\label{sec:proof2}
The main novelty in Theorem \ref{thm:sharpness} is the following differential inequality. The proof mimics the approach of \cite[Lemma 2.6]{DCT} for the Ising model with the random current replaced by the double cluster swapping from Section \ref{sec:therep}. We will prove an inequality of the form
\begin{equation}\label{eq:roughineq}
	\frac{\d \left< \vp_o \right>_{G_n,\beta,h}}{\d \beta} \geq c_n \times \E^{+ / +}_\beta \left[ \vp_\beta(\S) \i(o \in \S) \right],
\end{equation}
where $\S$ is a \emph{random set} defined in terms of the $(\w^+,\w^-)$ representation. The inequality itself appears from an application of the FKG inequality. 
The argument is technically closer to simplified but weaker version than in \cite{DCT} (unpublished but reproduced in \cite[Appendix]{lis2019circle}) since we lower bound the derivative of $\left< \vp_o \right>_{G_n,\beta,h}$, rather than that of its square $\left< \vp_o \right>_{G_n,\beta,h}^{2}$, which is why the resulting exponent in \eqref{eq:exponent} is $1$ and not $1/2$. 
We must also handle the fact that the spins are continuous rather than discrete to stop the constant $c_n$ in \eqref{eq:roughineq} from degenerating to zero as $n \to \infty$.

\begin{proposition}[]\label{prop:mainineq}
	Let $P \in \Ec$ and $\eps > 0$. There exists $c = c(\eps) > 0$ such that for every $\beta \in (0,1/\eps)$, $h > 0$ and $n \geq 1$,
	\begin{equation}\label{eq:mainineq}
		\liminf_{n \to \infty} \frac{\d \left< \vp_o \right>_{G_n,\beta,h}}{\d \beta} \geq \frac{c}{\beta} \left( \inf_S \vp_\beta(S)  \right) \left( 1 - \frac{ 2\left< \vp_o \right>_{\beta}^+}{  \left< | \vp_o|  \right>_{\beta}^+}\right).
	\end{equation} 
\end{proposition}

Actually, we will prove the following finite-volume version of this statement:
\begin{equation}\label{eq:ineqfin}
	 \frac{\d \left< \vp_o \right>_{G_n,\beta,h}}{\d \beta} \geq \frac{c}{\beta} \times ( \inf_S \vp_\beta(S))  \left( 1 - \frac{2 \left< \vp_o \right>_{G_n,\beta,h}}{  \left< | \vp_o|  \right>_{G_n,\beta,h}}\right) - \varepsilon_\beta(G_n,h),
\end{equation} 
where 
\[
	\varepsilon_\beta(G_n,h) = \sum_{\substack{u \in G_n \\ v \in \G \setminus G_n}} J_{uv} \left< \vp_o; \vp_u \right>_{G_n,\beta,h}.
\]
Here and from now on, we use the standard notation for truncated correlations e.g.~$\left< \vp_0; \vp_u \right> = \left< \vp_o \vp_u \right>_{}^{} - \left< \vp_o \right>_{}^{} \left< \vp_u \right>_{}^{}$.
It can be argued that $\varepsilon_\beta(G_n,h) \to 0$ as $n \to \infty$ by using the GHS inequality (see \cite[Remark 2.3]{DCT}), which is valid for $P \in \Ec$. Therefore, Proposition \ref{prop:mainineq} follows from \eqref{eq:ineqfin} by taking $n \to \infty$.

\begin{proof}[Proof of Proposition \ref{prop:mainineq}]
	We start by expressing the left-hand side of \eqref{eq:ineqfin} in terms of certain truncated correlations
	{\hypersetup{linkcolor=black}\footnote{Note that we have symmetrised in order to sum over $u,v \in G_n$ rather than $uv \in E(G_n)$}}
	\[
		\frac{\d \left< \vp_o \right>_{G_n,\beta,h}} {\d \beta} = \frac{1}{2} \sum_{u,v \in G_n} J_{uv}\left< \vp_o; \vp_u \vp_v \right>_{G_n,\beta,h}.
	\] 
	Using our representation and the identity \eqref{eq:corr3}, this is equal to
	\[
		\sum_{u,v \in G_n} J_{uv} \E^{+ / +}_{\beta} \left[ |\vpm_o \vpm_u\vpp_v| \i(o \overset{\w^-}{\longleftrightarrow} u, v \overset{\w^+}{\longleftrightarrow} g) \right].
	\] 
	We define the set $\S$ by\footnote{Note that $\S$ is possibly non-connected.} 
	\[
		\S=\Big\{ v \in G_n \mid v \overset{\w^+}{\centernot{\longleftrightarrow}} \g 
		\Big\}.
	\]
	See Figure \ref{fig:Sdef2} for an illustration of the proof.
	\begin{figure}[t]
	\centering
	\includegraphics[scale=0.7]{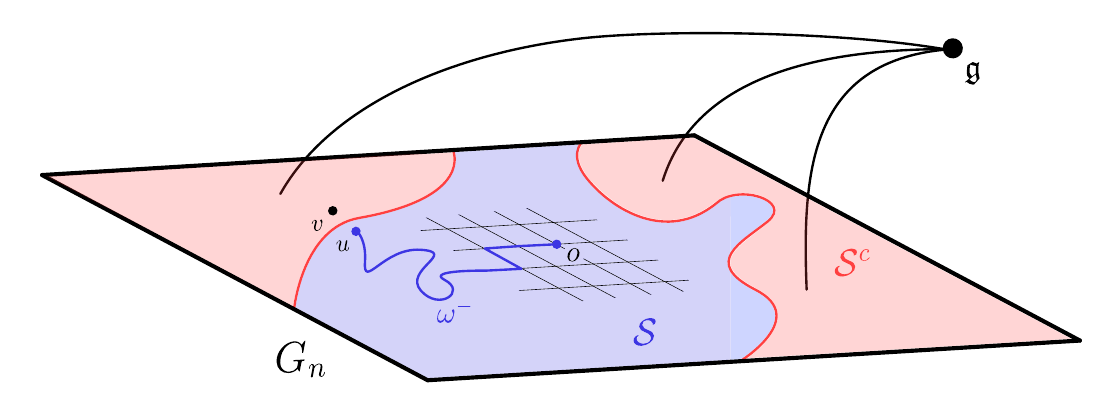}
	\caption{A illustration of the notation and setting in the proof of Proposition \ref{prop:mainineq}.}\label{fig:Sdef2}
\end{figure}
	We obtain a lower bound by only summing over $u \in \S$ and further requiring $o$ to be connected to $u$ inside $\S$,
	\begin{equation}\label{eq:midstep}
		\frac{\d \left< \vp_o \right>_{G_n,\beta,h}} {\d \beta} \geq 
		\E^{+ / +}_{\beta} \Bigg[ \sum_{\substack{u \in \S \\ v \not\in \S}} J_{uv}|\vpm_o \vpm_u\vpp_v| \i(o \underset{\S}{\overset{\w^-}{\longleftrightarrow}} u) \Bigg] =: A - B,
	\end{equation}
	where we have added and subtracted the terms corresponding to $v \in \G \setminus G_n$ to define
	\begin{equation}
	A = \E^{+ / +}_\beta\left[ \sum_{u \in \S} |\vpm_o\vpm_u|\i(o \underset{\S}{\overset{\w^-}{\longleftrightarrow}} u)\Big( \sum_{v \in G_n \setminus \S} J_{uv}|\vpp_v| + \sum_{v \not\in G_n} J_{uv}\Big)\right] 
\end{equation}
and
\begin{equation}
	B =  \E^{+ / +}_\beta\left[\sum_{\substack{u \in \S \\ v \not\in G_n}} |\vpm_o\vpm_u|\i(o \underset{\S}{\overset{\w^-}{\longleftrightarrow}} u) \right]. 
	\end{equation} 
	\paragraph{Step 1. }We now claim that conditionally on $\S$ \emph{and on} $\L_{\S^c} := ( \L^+_v, \L^-_v)_{v \not\in \S}$, we have
	\[
		\E_\beta^{+ / +}\left[|\vpm_o\vpm_u| \i(o \underset{\S}{\overset{\w^-}{\longleftrightarrow}} u) \mid \S, \L_{\S^c} \right] \geq \frac{1}{2}\left< \vp_o \vp_u \right>_{\S,\beta,0},
	\] 
	for any $o,u \in \S$.
	\\[1em]
	\noindent This is the main application of the FKG inequality. For any $S \subset G_n$, let $\Q_{S,\beta}$ be the coupling of $(\L^+_S,\L^-_S,\w^+_S,\w^-_S)$ arising from the measure $\left< \cdot \right>_{S,\beta,0} \otimes \left< \cdot \right>_{S,\beta,0}$. We will show 
	\begin{align}\label{eq:keystep}
		\E_{\beta}^{+ / +}\left[ |\vpm_o \vpm_u| \i(o \underset{\S}{\overset{\w^-}{\longleftrightarrow}} u)\mid \S , \L_{S^c}  \right] \geq \Q_{\S,\beta}\left[ |\vpm_o \vpm_u| \i(o \overset{\w^-}{\longleftrightarrow} u) \right]  = \frac{1}{2} \left< \vp_o \vp_u \right>_{\S,\beta,0},
	\end{align} 
	where the equality is due to \eqref{eq:corr1}. The inequality follows from a stochastic domination between the two measures, which as in the proof of Proposition \ref{prop:FKG} we need only prove for $(\L^+,\L^-)$. Indeed, by using \eqref{eq:rotden3} we see that the relative density of $(\L^+_\S,\L^-_\S)$ in $\P_\beta^{+ / +}(\cdot \mid \S,\L_{\S^c})$ to $\Q_{\S,\beta}$ is proportional to 
	\[
		\exp\Big( 2\beta \sum_{\substack{u \in \S \\ v \not\in S}} (\L^-_u \L^-_v - \L^+_u\L^+_v) - 2(h,\L^+)_S\Big) \times \frac{\sum_{\w^- \subset E(G_n)} 2^{k(\w)} q^{G_n}_{\L^-}(\w^-) }{\sum_{\w^- \subset E(G_n)} \i(\S \overset{\w^-}{\centernot{\longleftrightarrow}} \S^c) 2^{k(w)} q^{G_n}_{\L^-}(\w^-) }.
	\] 
	This expression is decreasing as a function of $(\L^+_\S,-\L^-_\S)$, which for the second factor is due to the standard monotonicity of the random-cluster model \cite{Grimmett}. Using Lemma \ref{lemma:stochdom}, the FKG property of $\Q_{\S,\beta}$ is all that is needed to conclude.  
	\paragraph{Step 2. } Our second claim is that there exists a constant $c > 0$ (only depending on $\G$, $P$ and $\eps$) such that conditionally \emph{on just} $\S$,
	\[
		\E^{+ / +}_\beta\left[ |\vp^+_v| \mid \S \right] \geq 2c,
	\] 
	for any $\beta \geq 0$ and any $v \not\in \S$. Note that on the event $v \not\in \S$ it must be that $|\vp^+_v| > 0$ as otherwise no edge incident to $v$ can be open in $\w^+$ . Thus the claim is immediate if the potential $P$ is of Ising type, and otherwise it is the content of Lemma \ref{lemma:technical}.
	\paragraph{Step 3.} By conditioning on $\S$ and $\L_{\S^c}$ and then applying Step 1 to the term $A$, we get
	\begin{equation*}
		A \geq \frac{1}{2} \times \E^{+ / +}_\beta\left[ \sum_{u \in \S} \left< \vp_o\vp_u \right>_{\S,\beta,0}( \sum_{v \in G_n \setminus \S} J_{uv}|\vpp_v| + \sum_{v \not\in G_n} J_{uv})\right]
	\end{equation*}
	Then, conditioning on just $\S$ and applying Step 2, the same term is lower bounded by 
	\[
		c \times \E^{+ / +}_{\beta} \Bigg[ \i(o \in \S) \sum_{\substack{ u \in \S \\ v \in \G \setminus \S}}J_{uv} \left< \vp_o \vp_u \right>_{\S,\beta,0} \Bigg] = c \times \E^{+ / +}\Bigg[ \i(o \in \S) \vp_\beta(\S) \Bigg].
	\] 
	Now, we simply note that for the second term in \eqref{eq:midstep}, we have
	\[
		B \leq  \E^{+ / +}_\beta\left[ \sum_{\substack{ u \in G_n \\ v \in \G \setminus G_n}} J_{uv} |\vpm_o\vpm_u |\i(o \overset{\w^-}{\longleftrightarrow} u)\right],
	\] 
	which is exactly equal to $\eps_\beta(G_n,h)$ by \eqref{eq:corr2}. All together, 
	\[
		\frac{\d \left< \vp_o  \right>_{G_n,\beta,h}}{\d \beta} \geq c \times \E^{+ / +}_\beta \left[ \i(o \in \S) \vp_\beta(S) \right] - \eps_\beta(G_n,h).
	\] 
	\paragraph{Step 4. } It only remains to show that
	\[
		\P^{ + / +}_\beta(o \in \S) \geq 1 - \frac{2\left< \vp_o \right>_{G_n,\beta,h}}{\left<|\vp_o| \right>_{G_n,\beta,h}}.
	\] 
	Indeed, by the FKG inequality (Proposition \ref{prop:FKG}) for $(\L^+,\w^+)$, we have
	\[
		\P^{ + / +}_\beta(o \not\in \S) = \P^{+ / +}_\beta(o \overset{\w^+}{\longleftrightarrow} \g) \leq \frac{\E^{+ / +}[ |\vpp_o|\i(o \overset{\w^+}{\longleftrightarrow} \g)]}{\E^{+ / +}\left[ |\vpp_o| \right] }
	\]
	The numerator is exactly $\left< \vp_o \right>_{G_n,\beta,h}$ and the denominator is at least{\hypersetup{linkcolor=black}\footnote{Here we used that if $X$ and $Y$ are symmetric independent random variables then $\mathbb{E}|X+Y| \geq \mathbb{E}|X|$.}} $ \left< |\vp_o| \right>_{G_n,\beta,h} / 2$, completing the proof of \eqref{eq:ineqfin}.
\end{proof}

For completeness we end this section by recalling how the dichotomy in Theorem \ref{thm:sharpness} arises from Proposition \ref{prop:mainineq}. The argument is very close to \cite{DCT} with only a few minor changes. 

\begin{proof}[Proof of Theorem \ref{thm:sharpness}]
	We consider a critical value given by
	\[
	\widetilde{\beta_c} = \sup \left\{  \beta \geq 0 \mid \exists S \text{ such that } \vp_\beta(S) < 1 \right\} 
	\] 
	Under the assumptions on $J$ and $P$ (notably \eqref{eq:Jasum}), it is clear that $ \widetilde{\beta_c}$ is strictly positive. We note that the disordered phase must occur for any $\beta < \widetilde{\beta_c}$ by Lemma \ref{lemma:expdecay}. Therefore, it only remains to consider the case when $ \widetilde{\beta_c} < \infty$ and to prove that there exists $ \widetilde{c} > 0$ such that
	\begin{equation}\label{eq:tildineq}
	\left< \vp_o \right>_{\beta}^{+} \geq \widetilde{c} \times \frac{\beta - \widetilde{\beta_c}}{\beta}
	\end{equation}
	for all $\beta \geq \widetilde{\beta_c}$. 
	
	Choose any $\eps > 0$ so that $ \widetilde{\beta_c} < 1/\eps$, thus there is $C = C(\eps) > 0$ from Lemma \ref{lemma:massive} so that
	\[
		\liminf_{n \to \infty} \frac{\d \left< \vp_o \right>_{G_n,\beta,h}}{\d \beta} \geq \frac{c}{\beta} \left( 1 - \frac{ 2\left< \vp_o \right>_{\beta}^+}{C}\right).
	\] 
	for any $ \widetilde{\beta_c} \leq \beta \leq 1 / \eps$. Integrating this inequality and taking $n \to \infty$, we see that 
	\[
		\left< \vp_o \right>_{\beta}^{+} \geq\Big ( \frac{c}{2} \times \frac{\beta - \widetilde{\beta_c}}{ \beta}\Big) \wedge \frac{C}{4} \quad \text{ for all $ \widetilde{\beta_c} \leq \beta \leq 1 / \eps.$} 
	\] 
	which implies \eqref{eq:tildineq} for all $\beta \geq \widetilde{\beta_c}$ with $\widetilde{c} = (C / 4) \wedge (c \times (1 - \eps \widetilde{\beta_c}) / 2)$.
\end{proof}
\section{A technical lemma}\label{sec:technical}

Before proceeding, we state an estimate that we will use frequently (but which is only relevant when the support of $\l^{(P)}$ is unbounded). 
It is an immediate consequence of Assumption \ref{asum:two} and the Griffiths inequalities.  
\begin{lemma}\label{lemma:massive}
	Let $\eps > 0$. Then, there exists $C = C(\eps,P,\G)$ such that 
	\begin{equation}\label{eq:massive}
	\left< |\vp_o| \right>_{\G,\beta,h} \leq C
	\end{equation} 
	for all $0 \leq \beta \leq 1 / \eps$ and $0 \leq h \leq 1$.
\end{lemma}

The goal of this section is to prove the following lemma where we stress again that the conclusion is immediate if the potential $P$ is of Ising type (as explained in the proof of Proposition \ref{prop:mainineq}).
Therefore, from now on we may assume that $P$ is as in \eqref{eq:Pdef}. 
We assume the setting and notation in the proof of Proposition \ref{prop:mainineq}. 
\begin{lemma}\label{lemma:technical}
	There exists $c = c(\eps) > 0$ such that for every $\beta \in (0,1/\eps)$ and each $n \geq 1$ and finite set $S \subset G_n$ with $v \not\in S$, 
	\[
		\E^{+ / +}_{\beta,h}\left[ \, |\vp^+_v| \mid \S = S \right] \geq c.
	\]
\end{lemma}
It is key that the constant $c$ does not depend on $n \geq 1$. We are faced with the problem of proving a rather simple estimate but under a complicated conditional measure. Our strategy will be to rely on the FKG theory to prove a suitable stochastic domination comparing this complicated measure with a simpler measure. We are then left with the easier task of proving this simple estimate under the (relatively) simple measure. We start by defining said measure.
\\[1em]
\noindent We begin with some notation. Fixing $S$ from now, we define $X = S^c$ and decompose
\[
\L^+ = (\L^+_S,\L^+_X)
\] 
where $\L^+_S = (\L^+_u)_{u \in S}$ and $\L^+_{X} = (\L^+_x)_{x \in X}$. We will also heavily use the notation from Section \ref{sec:corrineq}.
The most important tool in the proof is the following auxillary measure. Its definition is best motivated by the form of the relative density \eqref{eq:relden}. Indeed, we have essentially dropped each factor in \eqref{eq:compden} which is increasing in $\L^+_X$, leaving us with the definition of $\mu_{\beta;S}^{(n)}$ in \eqref{eq:mudef} which we can then prove the stochastic domination for.
\begin{definition}[]
	We define a measure $\mu_{\beta;S}^{(n)}$ on $(\L^+,\L^-) \in \left[0,\infty \right)^{G_n} \times \left[ 0,\infty \right)^{G_n}$ via
	\begin{multline}\label{eq:mudef}
		\d\mu_{\beta;S}^{(n)}(\xi^+,\xi^-) \propto Z^{G_n}(\L^-)  Z^{S}(\L^+) \times \exp\Big( 2\beta(\L^-,J\L^-) + 2\beta(\L^+,J\L^+)_S - 2\beta \sum_{\substack{x \in X \\ y \in S}} J_{xy} \L^+_x \L^+_y\Big) \\ \times \exp(- \textstyle\sum_{x \in X}\I(\L^+_x,\L^-_x)) \d\l_2^X(\L^+) \d\l_2^X(\L^-) \times \d\l_2^S(\L^+)\d\l_2^S(\L^-).
	\end{multline} 
\end{definition}
\noindent The measure $\mu_{\beta;S}^{(n)}$ has strong decoupling properties which we will make use of:
\begin{enumerate}[(1)]
	\item Conditionally on $(\L^-,\L^+_S)$, the collection
		\[
			\L^+_X = (\L^+_x)_{x \in X}
		\] 
		is independent under $\mu_{\beta;S}^{(n)}$.
	\item Conditionally on $\L^+_X$, $\L^-$ and $\L^+_S$ are independent under $\mu_{\beta;S}^{(n)}$.  
\end{enumerate}
These two properties are consequences of the density in \eqref{eq:mudef}, in particular, $\L^-$ and $\L^+_X$ are linked only through the factor $\exp(-\I(\L^+_x,\L^-_x)$ (which only appears for $x \in X$) and $\L^+_S$ and $\L^+_X$ are only linked through the term $-2\beta\sum_{x \in X, y \in S}\L^+_x \L^+_y$. This is illustrated in Figure \ref{fig:mudef}.

\begin{figure}
	\centering
	\includegraphics[scale=1.6]{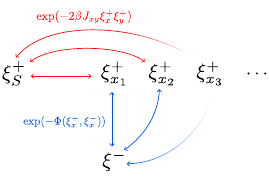}
	\caption{A representation of the dependencies in $\mu_{\beta;S}^{(n)}$: The vertices $x_1,x_2,x_3,\ldots$ enumerate $X$ and each drawn arrow represents a factor in \eqref{eq:mudef} involving both endpoints. Moreover, these factors induce a `negative correlation' in the sense that the corresponding terms in \eqref{eq:mudef} satisfy the lattice condition with respect to the order $(\L^+_X,-\L^+_S,-\L^-)$. As $\L^+_S$ and $\L^-$ are both negatively correlated with $\L^+_X$, they end up being positively correlated with one another.}
	\label{fig:mudef}
\end{figure}

We now prove the lemma comparing the complicated measure $\P^{+ / +}_{\beta}( \cdot \mid \S = S)$ with the simpler measure $\mu_{\beta;S}^{(n)}$.
\begin{lemma}\label{lemma:prelimlem}
In the setup of Lemma \ref{lemma:technical}, we have $\E^{ + / +}_{\beta,h}\left[ \, |\vp^+_v| \mid \S = S \right] \geq \mu_{\beta;S}^{(n)}\left[ \L^+_v \right].$ \end{lemma}
\begin{proof}
	The density of $(\L^+,\L^-)$ under $\P^{+ / +}\left[\,\cdot \mid \S = S \right]$  is proportional to
	\begin{multline}\label{eq:compden}
		Z^{G_n}(\L^-) Z^{S}(\L^+) \Big( \sum_{\substack{ \w \subset \overline{E}(X) \\ \w \text{ spanning}}} q_{\L^+}^{ \overline{X}}(\w) \Big )  \times \exp \big( 2\beta(\L^-,J\L^-) + 2\beta(\L^+,J\L^+)_S + 2\beta(\L^+,J\L^+)_X\big) \\ \times \exp\Big( - 2\beta \textstyle\sum_{\substack{x \in X \\ y \in S}} J_{xy} \L^+_x \L^+_y - 2(h,\L^+)_S + 2(h,\L^+)_X\Big)  \times \exp\big(- \sum_{u \in V} \I(\L^+_u,\L^+_u)\big) \d\l^G_2(\L^+) \d\l_2^G(\L^-).
	\end{multline}
	Here, the first sum is over the possible realisations of $\w^+|_{ \overline{E}(X)}$ on the event $\S = S$ i.e.~subsets of the edges in $ \overline{E}(X)$ which connect every vertex in $X \cup \left\{ \g \right\} $.
	Thus, the relative density of $(\L^+,\L^-)$ in $ \P^{+ / +}\left[ \,\cdot \mid \S = S \right]$ to $\mu_{\beta;S}^{(n)}$ is proportional to
	\begin{equation}\label{eq:relden}
		\underbrace{ \sum_{\substack{\w \subset \overline{E}(X) \\ \w \text{ spanning}}} q_{\L^+_X}^{ \overline{X}}(\w) }_{\text{(I)}} \times \underbrace{ \exp \big( 2\beta(\L^+,J\L^+)_X -2(h,\L^+)_S + 2(h,\L^+)_X\big)}_{\text{(II)}}  \times \underbrace{\exp\big( - \textstyle\sum_{u \in S}\I(\L^+_u,\L^-_u) \big)}_{\text{(III)}}. 
	\end{equation} 
	
	\noindent\emph{Claim 1. } The triple $(\L^+_X,-\L^+_S,-\L^-)$ satisfies the FKG inequality under $\mu_{\beta;S}^{(n)}$.
	\\[1em]
	We stress that this is a \emph{completely different ordering} to beforehand. In particular, $\L^+_S$ and $\L^-$ are now positively associated with each other. This is because, as shown in Figure \ref{fig:mudef}, $\L^-$ is negatively correlated with $\L^+_X$, due to the presence of the factor $\exp(-\I(\L^+_x,\L^-_x))$ for $x \in X$ (like before), and $\L^+_S$ is also negatively correlated with $\L^+_X$ because of the negative sign in the factor $\exp( - 2\beta\sum_{x \in X, y \in S} \L^+_x \L^+_y)$.
	These observations are all that is needed to verify the FKG lattice condition, using entirely the arguments as for Proposition \ref{prop:FKG} except with care now for the signs of each term in the exponential. 
	\\[1em]
	\noindent\emph{Claim 2. }The relative density in \eqref{eq:relden} is increasing with respect to $(\L^+_X,-\L^+_S,-\L^-)$.
	\\[1em]
	Indeed, this is most immediate for term (II), noting that $(\L^+,J\L^+)_X$ is increasing in $\L^+_X$ since $\L^+ \geq 0$. That term (I) is increasing is due to the fact that the edge parameters
	\[
		(1 - \exp( -4 \beta J_{xy} \L^+_x \L^+_y))_{xy \in E(X)} \text{ and } (1 - \exp( - 2h_x \L^+_x))_{x \in X}
	\] 
	are increasing in $\L^+_X$ and that the event $\left\{  \w \text{ spanning} \right\}$ is an increasing event, so the map
	\[
	\L^+_X \mapsto \sum_{\substack{\w \subset \overline{E}(X) \\ \w \text{ spanning}}} q_{\L^+_X}(\w)
	\] 
	is increasing by the stochastic monotonicity of (inhomogeneous) Bernoulli percolation on $ \overline{E}(X)$.
	Finally, the term (III) is increasing with respect to $(-\L^+_S,-\L^-)$ by Lemma \ref{lemma:Iprops}.
	\\[1em]
	Together, Claim 1 and 2 and Lemma \ref{lemma:stochdom} imply that there is a stochastic domination
	\[
		\P^{+ / +}(\cdot \mid \S = S) \overset{(\L^+_X,-\L^+_S,-\L^-)}{\succeq} \mu_{\beta;S}^{(n)},
	\] 
	which completes the proof.
\end{proof}

\begin{proof}[Proof of Lemma \ref{lemma:technical}]
	With Lemma \ref{lemma:prelimlem} in hand, it suffices to show that
	\begin{equation}\label{eq:muest}
		\mu_{\beta;S}^{(n)}\left[ \L^+_v  \right]  \geq c > 0
	\end{equation} 
	for each $n \geq 1$, $S \subset G_n$ and $v \not\in  S$.
	\\[1em]
	\noindent Let us first assume $\l$ has unbounded support, i.e.~$P$ is as in \eqref{eq:Pdef} with $A = \infty$. We will use the properties of $\mu_{\beta;S}^{(n)}$ to argue that \eqref{eq:muest} would follow if we knew  
	\begin{equation}\label{eq:upperbound}
		\mu_{\beta;S}^{(n)}\Bigg[ \sum_{\substack{u \in S : \,uv \in E}} J_{uv} \L^+_u \Bigg] \leq C \quad\text{ and }\quad \mu_{\beta;S}^{(n)}\left[ \L^-_v \right] \leq C'
	\end{equation} 
	for some $C,C' < \infty$. 
	Indeed, define
	\[
		F(\alpha,\gamma) = \frac{\int_{(0,\infty)} x \exp( - \gamma x - \I(x,\alpha)) \d\l(x)}{\int_{(0,\infty)} \exp( -\gamma x - \I(x,\alpha)) \d\l(x)}
	\] 
	for $\alpha,\gamma \geq 0$. Then, by the first decoupling property of $\mu_{\beta;S}^{(n)}$, we have that
	\[
		\mu_{\beta;S}^{(n)}\left[ \L^+_v \right] = \mu_{\beta;S}^{(n)}\Big[ F(\L^-_v,\sum_{\substack{u \in S : \,uv \in E}} 2\beta J_{uv} \L^+_u) \Big]. 
	\] 
	Therefore{\hypersetup{linkcolor=black}\footnote{Here, we used that if two non-negative random variables satisfy $\E X \leq C$ and $\E Y \leq C'$ then $\P(X \leq 3C, Y \leq 3C') \geq 1 / 3$. }},
	\[
		\mu^{(n)}_{\beta;S}\left[ |\vpp_v| \right] \geq \frac{1}{3} \inf \left\{ F(\alpha,\gamma) \mid \alpha \leq 3C', \gamma \leq {6C} / \eps  \right\}
	\] 
	where $C$ and $C'$ are the constants appearing in \eqref{eq:upperbound}. It is not hard to see that this lower bound (which only depends on $C,C'$ and $P$) is strictly positive as the map $F$ itself is strictly positive and continuous in $\alpha$ and $\gamma$.

	It remains to prove \eqref{eq:upperbound} for which we use one final stochastic domination. Let $(\vps,\vps') \sim \left< \cdot \right>_{S,2\beta,0,2P} \otimes \left< \cdot \right>_{G_n,2\beta,0,2P}$ i.e.~sampled according to the measure with
	\[
		\beta \mapsto 2\beta \quad\text{ and }\quad P \mapsto 2P
	\]
	and on $S$ and $G_n$ respectively. We next claim that for the marginal $(\L^+_S,\L^-) \sim \mu_{\beta;S}^{(n)}$, there is a stochastic domination
	\begin{equation}\label{eq:twodom}
		(\L^+_S,\L^-) \preceq (|\vps|,|\vps'|).
	\end{equation}
	Indeed, the relative density is simply proportional to the integral over $\L^+_X$,  
	\begin{equation}\label{eq:rel2}
		\int_{(0,\infty)^{X}} \exp\Big( -2\beta \textstyle\sum_{\substack{u \in S \\ x \in X}} J_{ux} \L^+_u \L^+_x - \textstyle\sum_{x \in X} \I(\L^+_x,\L^-_x)\Big)  \d\l_2^X(\L^+),
	\end{equation}
	The integrand in \eqref{eq:rel2} is increasing in $(\L^+_S,\L^-)$ by again using Lemma \ref{lemma:Iprops}, and so combined with the FKG inequality for $\mu_{\beta;S}^{(n)}$ we have the claimed domination.  
	Using this, \eqref{eq:upperbound} follows by applying Lemma \ref{lemma:massive} to the potential $2P$ and with  $\eps/2$ to obtain $C,C' > 0$ which are uniform in $\beta \leq 1 / \eps$, $n \geq 1$, $S \subset G_n$ and finally $v \in S$.
	\\[1em]
	\noindent Only a few simple adaptations are needed if $\l$ has bounded support, i.e.~$0 < A < \infty$ in \eqref{eq:Pdef}. Let $M$ be the supremum of the support of $\l$. We replace \eqref{eq:upperbound} by two estimates, noting that the first is deterministic, 
	\begin{equation}\label{eq:upperbound2}
		\sum_{\substack{u \in S \\ uv \in E}} J_{uv} \L^+_u  \leq C \quad\text{ and }\quad \mu_{\beta;S}^{(n)}\left[ \L^-_v \right] \leq (1-\delta)M
	\end{equation} 
	for some $C < \infty$ and $\delta > 0$. This modification is then sufficient to prove \eqref{eq:muest} by arguing exactly as above{\hypersetup{linkcolor=black}\footnote{Here, we instead use that if two random variables $X \leq C$ and $0 \leq Y \leq M$ satisfy $\E Y \leq (1-\delta)M$ then $\P(X \leq C, Y \leq (1-\delta)M ) \geq \delta$.}} and the fact that
	\[
		\delta \times \inf \left\{ F(\alpha,\gamma) \mid \alpha \leq (1-\delta)M, \gamma \leq {2C} / \eps  \right\}
	\] 
	is strictly positive for any $C < \infty$ and $\delta > 0$, for as before $F$ is continuous and $F(\alpha,\gamma) > 0$ for any $\gamma \geq 0$ and $\alpha < A$ (but not for $\alpha = A$).
	The first bound in \eqref{eq:upperbound2} is true by taking $C = \left( \sum_{x} J_{ox} \right) M$ and the second is true by the same domination argument above, and using that the density of $\vps_v$ under $\left< \cdot \right>_{G_n,2\beta,0,2P}$ is clearly bounded above and below by (with the same $C$)
	\[
	\exp( -2\beta C \vps_v) \d\l_2(\vps_v) \quad \text{ and }\quad \exp(2\beta C \vps_v) \d\l_2(\vps_v)
	\] 
	respectively, so that $\left< |\vps_v| \right>_{G_n,2\beta,0,2P} \leq (1-\delta)M$ for some $\delta = \delta(C,\eps)$.
\end{proof}
\begin{appendices}
\section{Proof of Proposition \ref{prop:main}} \label{sec:appendix}
This section is independent of the remainder of the paper and aims to prove Proposition \ref{prop:main}: 
if $G$ is a fully reduced graph with marked set of vertices $W$ and satisfied the $W$-crossing property, 
then it can be drawn in the unit disk with the vertices of $W$ along the boundary. 
Throughout, we assume the setting and definitions of Section \ref{sec:graphred}.
\\[1em]
\noindent Before starting, we need a classical result in graph theory, known as Menger's theorem \cite{Menger}, 
see \cite{Goring} for a particularly short proof. 

\begin{lemma}[Menger's theorem] \label{lemma:Mengerstheorem}
    Let $G=(V,E)$ be a finite graph and let $A, B \subseteq V$. 
    Then the minimal size of an  $(A, B)$-separator is equal to the number of vertex-disjoint paths connecting $A$ and $B$. 
\end{lemma}

\begin{proof}[Proof of Proposition \ref{prop:main}]
    We will apply induction on the number of vertices of $G = (V, E)$. 
    For the base case, $|V| = |W| = 4$, the result is trivial. 
    
    Assume the result holds for all finite graphs with $ k - 1$ or less vertices.
    Let $G = (V, E)$ be a finite graph with $k$ vertices, which is fully reduced. 
    If $|V| = |W|$, the result is an immediate consequence of the $W$-crossing property and we are done. 
    Thus, we may suppose without loss of generality that there is some vertex $x \in V \setminus W$. 
    Pick any such vertex and an edge $e = \{x, y\}$ emanating from it.
    
    Write $G / e$ for the graph obtained from contracting $e $, let $v_e$ be the obtained vertex. 
    Then $G / e$ has $k - 1$ vertices, hence the induction hypothesis implies that either $G / e$ can be reduced, 
    or that $G / e$ can be drawn in the disk. We handle the two cases separately. 
	\\
	\\
	\noindent \textbf{Case 1}. 
	$G / e$ can be reduced, with reduction set $S$.
    Write $X$ for the (non-empty) maximal disconnected set. 
    
    Notice that $v_e \in S$, as otherwise $S$ would be a reduction set for $G$, but $G$ is fully reduced by assumption. 
    A similar argument gives $|S| = 3$. 
    Hence, in the graph $G$, the set $P = (S \setminus \{v_e\}) \cup \{x, y\}$ has size four and $P$ is a $(W, X)$ separator. 
    
    Note that $P$ may contain elements of $W$, but by construction is not equal to $W$. 
    Let $Y$ be the connected component of $W$ in $G - P$, so $V$ is the disjoint union of $X, P$ and $Y$. 
    Thus, by construction, both $X$ and $Y$ are non-empty. 
    
    Write $G_X$ for the subgraph of $G$ induced by $X \cup P$ and $G_Y$ for the one induced by $Y \cup P$. 
    
    By Menger's theorem, there exist $4$ disjoint paths connecting $W$ to $P$ in $G - P$, 
    since $G$ is fully reduced. 
    As such, the ordering of elements in $W$ induces and ordering of $P = (p_1, p_2, p_3, p_4)$, 
    and the graph $G_X$ must satisfy the $P$-crossing property. 
    Moreover, $G_X$ is reduced because $G$ is reduced by assumption. 
    Since $|X \cup P| < |V|$, the induction hypothesis implies that $G_X$ can be drawn in the disk with $P$ on the boundary of the disk. 
    
    Write next $G_Y'$ for the graph obtained by adding to $G_Y$ edges between elements of $P$ in cyclic order. 
    We make the following claim. 
    \\
    \\
    \noindent \textbf{Claim:}
    The graph $G_Y'$ is fully reduced and satisfies the $W$-crossing property. 
    \\
    \\
    \noindent Let us first argue how to conclude the proof of Proposition \ref{prop:main} from this claim. 
    Since the number of vertices of $G_Y'$ is less or equal to $k -1$, 
    the claim and the induction hypothesis imply that $G_Y'$ can be drawn in the disk. 
    The vertices of $P$ lie on one cycle, hence they must form a face in the drawing. 
    Indeed, consider four disjoint paths ${\gamma}_1, \ldots, {\gamma}_4$ connecting 
    $x_1, \ldots, x_4$ to $p_1, \ldots, p_4$ inside $G_Y'$ 
    If $P$ is not a face, then we must have (up to relabeling) 
    that the edge connecting $p_1, p_2$ either intersects one of the paths ${\gamma}_1, \ldots, {\gamma}_4$, 
    intersects one of the other edges $p_i, p_{i+1}$ or must go outside of the disk 
    (since $W$ lies on the boundary of the disk). 
    None of this can happen because $G_Y'$ is drawn in the disk with $W$ on the boundary, so $P$ is a face. 
    We can put $G_X$ inside this graph by gluing $G_X$ and $G_Y$ at the vertices of $P$ 
    and the obtained graph is planar and can be drawn in the disk. 
    Finally, removing the edges in the cycle of $P$ from this graph, we get back $G$ and hence for Case 1, 
	the induction is finished.
    
    We are left to prove the claim. 
	Observe that $G_Y'$ is fully reduced as any reduction set would also be a reduction set of $G$. 
    
    For the $W$ crossing property, argue by contradiction. 
	Suppose there exist two simple, disjoint paths in $G_Y'$ violating the $W$-crossing property, 
	say $\gamma$ and $\widetilde{\gamma}$ and say they connect $x_1, x_3$ and $x_2, x_4$ in $W$ respectively. 
	Both $\gamma$ and $\widetilde{\gamma}$ use (at least) one edge added between the vertices of $P$. 
	Indeed, suppose that $\gamma$ does not use any of the edges added to $P$, 
	then it is a path in $G$. 
	Since $G_X$ is connected, also the path $\widetilde{\gamma}$ can be lifted to $G$ without intersecting $\gamma$, 
	contradicting the $W$-crossing property of $G$. 

	Thus, $\gamma$ and $\widetilde{\gamma}$ use an edge of $P$ and all elements of $P$, 
	and since $G_Y$ is planar and can be drawn in the disk with $W$ as a single face, the induced ordering on $P$ by $\gamma, \widetilde{\gamma}$
	viewed in $G_Y$ has to agree with the one induced by $W$ above. 
	This is a contradiction with the existence of the non-intersecting simple paths $\gamma$ and $\widetilde{\gamma}$,
	and as such proves the claim. 
	\\
	\\
	\textbf{Case 2:} $G / e$ is fully reduced. 
	By the induction hypothesis, $G / e$ can be drawn in the disk with $W$ on the boundary. 
	The only problem that can still occur is that when `unmerging' $xy$, there are edges that now must cross. 
	Consider the set $N_x$ of neighbors of $x$ in $G$. 
	The graph $G - x$ is planar since $G / e$ is planar, and thus can be drawn with $W$ on the boundary of the disk. 
	We claim that $N_x$ must lie on a face of $G - x$, in which case we are done. 
	Suppose not. 
	For every subset $A \subset N_x$ of size $3$, there must exist $3$ disjoint paths simple paths to $W$ in $G - x$, 
	since $G$ is fully reduced, and furthermore, there exists a further simple disjoint path from $x$ to $W$, 
	which exits $N_x$ at $a'$. 
	This induced an ordering of the elements in $A \cup \{a'\} = (a_1, \ldots, a_4)$. 
	
	If $N_x$ does not lie on a face of $G - x$, there hence exists an ordered set $A' = (a_1, \ldots, a_4)$ contained in $N_x$, 
	such that there are four disjoint paths $\tau_1, \ldots, \tau_4$ connecting $A'$ to $W$ \emph{and furthermore} such that
	(up to cyclic relabeling) there exists a further path in $G - x$ connecting $a_1$ and $a_3$ which only intersects $\tau_1, \ldots, \tau_4$ 
	at $a_1, a_3$. 
	But this readily implies that in $G$, the $W$-crossing property is violated, so $N_x$ must lie on a face of $G - x$,
	finishing the induction in Case 2 and therefore the proof of Proposition \ref{prop:main}. 
\end{proof}
\end{appendices}
\bibliography{DFK.bib}

@preamble{"\def\cprime{$'$} "}

@article{preston1974generalization,
	author = {Preston, C. J.},
	doi = {10.1007/BF01645981},
	journal = {Communications in Mathematical Physics},
	month = sep,
	number = {3},
	pages = {233--241},
	title = {A generalization of the {FKG} inequalities},
	volume = {36},
	year = {1974}}

@article{ruelle1976probability,
	author = {D. Ruelle},
	journal = {Communications in Mathematical Physics},
	number = {3},
	pages = {189 -- 194},
	publisher = {Springer},
	title = {{Probability estimates for continuous spin systems}},
	volume = {50},
	year = {1976}}

@article{guillarmou20262d,
	author = {Guillarmou, Colin and Gunaratnam, Trishen S and Vargas, Vincent},
	journal = {Communications in Mathematical Physics},
	number = {5},
	pages = {97},
	publisher = {Springer},
	title = {{2d Sinh-Gordon Model on the Infinite Cylinder}},
	volume = {407},
	year = {2026}}

@article{SinhGordon,
	author = {A.E. Arinshtein and V.A. Fateyev and A.B. Zamolodchikov},
	doi = {https://doi.org/10.1016/0370-2693(79)90561-6},
	issn = {0370-2693},
	journal = {Physics Letters B},
	number = {4},
	pages = {389-392},
	title = {{Quantum S-matrix of the (1 + 1)-dimensional Todd chain}},
	url = {https://www.sciencedirect.com/science/article/pii/0370269379905616},
	volume = {87},
	year = {1979}}

@article{easo2025counting,
	author = {Easo, Philip and Severo, Franco and Tassion, Vincent},
	doi = {10.1017/fmp.2025.10011},
	journal = {Forum of Mathematics, Pi},
	pages = {e23},
	title = {Counting minimal cutsets and $p_c<1$},
	volume = {13},
	year = {2025}}

@article{duminil2020existence,
	author = {Hugo Duminil-Copin and Subhajit Goswami and Aran Raoufi and Franco Severo and Ariel Yadin},
	doi = {10.1215/00127094-2020-0036},
	journal = {Duke Mathematical Journal},
	number = {18},
	pages = {3539 -- 3563},
	publisher = {Duke University Press},
	title = {{Existence of phase transition for percolation using the Gaussian free field}},
	url = {https://doi.org/10.1215/00127094-2020-0036},
	volume = {169},
	year = {2020}}

@article{glimm1975phase,
	author = {James Glimm and Arthur Jaffe and Thomas Spencer},
	journal = {Communications in Mathematical Physics},
	number = {3},
	pages = {203 -- 216},
	publisher = {Springer},
	title = {{Phase transitions for $\phi_{2}^{4}$ quantum fields}},
	volume = {45},
	year = {1975}}

@article{Lis22a,
	author = {Lis, Marcin},
	doi = {10.1093/imrn/rnaa380},
	issn = {1073-7928},
	journal = {International Mathematics Research Notices},
	number = {13},
	pages = {9909-9940},
	title = {{On Boundary Correlations in Planar Ashkin--Teller Models}},
	url = {https://doi.org/10.1093/imrn/rnaa380},
	volume = {2022},
	year = {2022}}

@article{Lis22b,
	author = {Marcin Lis},
	doi = {10.1214/22-EJP761},
	journal = {Electronic Journal of Probability},
	pages = {1 -- 21},
	publisher = {Institute of Mathematical Statistics and Bernoulli Society},
	title = {{Spins, percolation and height functions}},
	url = {https://doi.org/10.1214/22-EJP761},
	volume = {27},
	year = {2022}}

@article{RobSey,
	author = {Robertson, Neil and Seymour, P. D.},
	doi = {10.1016/0095-8956(90)90063-6},
	fjournal = {Journal of Combinatorial Theory. Series B},
	issn = {0095-8956,1096-0902},
	journal = {J. Combin. Theory Ser. B},
	mrclass = {05C10},
	mrnumber = {1056819},
	mrreviewer = {D.\ S.\ Archdeacon},
	number = {1},
	pages = {40--77},
	title = {Graph minors. {IX}.\ {D}isjoint crossed paths},
	url = {https://doi.org/10.1016/0095-8956(90)90063-6},
	volume = {49},
	year = {1990}}

@article{AHL,
	author = {{Alcalde L{\'o}pez}, Tom{\'a}s and Heeney, Lorca and Lis, Marcin},
	journal = {arXiv preprint arXiv:2602.05886},
	title = {{The Ising magnetisation field and the Gaussian free field}},
	year = {2026}}

@article{GPPS2,
	author = {Gunaratnam, Trishen S and Panagiotis, Christoforos and Panis, Romain and Severo, Franco},
	journal = {arXiv preprint arXiv:2501.05353},
	title = {{The supercritical phase of the $\varphi^4$ model is well behaved}},
	url = {https://arxiv.org/abs/2501.05353},
	year = {2025}}

@article{vEGPS,
	author = {Engelenburg, Diederik {van} and Gabran, Christophe and Panis, Romain and Severo, Franco},
	journal = {arXiv preprint arXiv:2510.23423},
	title = {{One-arm exponents of the high-dimensional Ising model}},
	url = {https://arxiv.org/abs/2510.23423},
	year = {2025}}

@article{DCP-2-point,
	author = {Duminil-Copin, Hugo and Panis, Romain},
	doi = {10.1007/s00220-025-05236-2},
	fjournal = {Communications in Mathematical Physics},
	issn = {0010-3616,1432-0916},
	journal = {Comm. Math. Phys.},
	mrclass = {60K35 (82B20 82B27)},
	mrnumber = {4865900},
	mrreviewer = {Rinaldo\ Schinazi},
	number = {3},
	pages = {Paper No. 56, 23},
	title = {New lower bounds for the (near) critical {I}sing and {$\varphi^4$} models' two-point functions},
	url = {https://doi.org/10.1007/s00220-025-05236-2},
	volume = {406},
	year = {2025}}

@article{ADC,
	author = {Aizenman, Michael and Duminil-Copin, Hugo},
	doi = {10.4007/annals.2021.194.1.3},
	fjournal = {Annals of Mathematics. Second Series},
	issn = {0003-486X,1939-8980},
	journal = {Ann. of Math. (2)},
	mrclass = {82B20 (60G60 82B27)},
	mrnumber = {4276286},
	mrreviewer = {Rongfeng\ Sun},
	number = {1},
	pages = {163--235},
	title = {Marginal triviality of the scaling limits of critical 4{D} {I}sing and {$\phi_4^4$} models},
	url = {https://doi.org/10.4007/annals.2021.194.1.3},
	volume = {194},
	year = {2021}}

@article{GRS,
	author = {Guerra, F. and Rosen, L. and Simon, B.},
	doi = {10.2307/1970988},
	fjournal = {Annals of Mathematics. Second Series},
	issn = {0003-486X},
	journal = {Ann. of Math. (2)},
	mrclass = {82.46 (81.46)},
	mrnumber = {378670},
	mrreviewer = {R.\ A.\ Minlos},
	pages = {111--189; ibid. (2) 101 (1975), 191--259},
	title = {The {${\bf P}(\phi)\sb{2}$} {E}uclidean quantum field theory as classical statistical mechanics. {I}, {II}},
	url = {https://doi.org/10.2307/1970988},
	volume = {101},
	year = {1975}}

@article{Goring,
	author = {G\"{o}ring, Frank},
	doi = {10.1016/S0012-365X(00)00088-1},
	fjournal = {Discrete Mathematics},
	issn = {0012-365X,1872-681X},
	journal = {Discrete Math.},
	mrclass = {05C40 (05C20)},
	mrnumber = {1761733},
	number = {1-3},
	pages = {295--296},
	title = {Short proof of {M}enger's theorem},
	url = {https://doi.org/10.1016/S0012-365X(00)00088-1},
	volume = {219},
	year = {2000}}

@article{Menger,
	author = {Menger, Karl},
	journal = {Fundamenta Mathematicae},
	language = {ger},
	number = {1},
	pages = {96-115},
	title = {{Zur allgemeinen Kurventheorie}},
	url = {http://eudml.org/doc/211191},
	volume = {10},
	year = {1927}}

@article{GPPS,
	author = {Gunaratnam, Trishen S and Panagiotis, Christoforos and Panis, Romain and Severo, Franco},
	journal = {arXiv preprint arXiv:2211.00319},
	title = {Random tangled currents for $\varphi^4$: translation invariant {G}ibbs measures and continuity of the phase transition},
	year = {2022}}

@article{LisT,
	author = {Lis, Marcin},
	doi = {10.1007/s10955-016-1690-x},
	id = {Lis2017},
	isbn = {1572-9613},
	journal = {Journal of Statistical Physics},
	number = {1},
	pages = {72--89},
	title = {{The Planar Ising Model and Total Positivity}},
	url = {https://doi.org/10.1007/s10955-016-1690-x},
	volume = {166},
	year = {2017}}

@article{LamOtt,
	author = {Lammers, Piet and Ott, S{\'e}bastien},
	journal = {arXiv preprint arXiv:2101.05139},
	title = {{Delocalisation and absolute-value-FKG in the solid-on-solid model}},
	year = {2021}}

@book{FriVel,
	author = {Friedli, S. and Velenik, Y.},
	doi = {10.1017/9781316882603},
	isbn = {978-1-107-18482-4},
	place = {Cambridge},
	publisher = {Cambridge University Press},
	title = {{Statistical Mechanics of Lattice Systems: A Concrete Mathematical Introduction}},
	year = {2017}}

@article{MMS,
	author = {Messager, A. and Miracle-Sole, S.},
	da = {1977/10/01},
	doi = {10.1007/BF01040105},
	id = {Messager1977},
	isbn = {1572-9613},
	journal = {Journal of Statistical Physics},
	number = {4},
	pages = {245--262},
	title = {{Correlation functions and boundary conditions in the Ising ferromagnet}},
	ty = {JOUR},
	url = {https://doi.org/10.1007/BF01040105},
	volume = {17},
	year = {1977}}

@article{DCT,
	author = {Duminil-Copin, Hugo and Tassion, Vincent},
	da = {2016/04/01},
	doi = {10.1007/s00220-015-2480-z},
	id = {Duminil-Copin2016},
	isbn = {1432-0916},
	journal = {Communications in Mathematical Physics},
	number = {2},
	pages = {725--745},
	title = {{A New Proof of the Sharpness of the Phase Transition for Bernoulli Percolation and the Ising Model}},
	ty = {JOUR},
	url = {https://doi.org/10.1007/s00220-015-2480-z},
	volume = {343},
	year = {2016}}

@article{Gin,
	author = {J. Ginibre},
	doi = {cmp/1103842172},
	journal = {Communications in Mathematical Physics},
	number = {4},
	pages = {310 -- 328},
	publisher = {Springer},
	title = {{General formulation of Griffiths' inequalities}},
	url = {https://doi.org/},
	volume = {16},
	year = {1970}}

@article{Lammers,
	author = {Lammers, P.},
	da = {2021/09/23},
	doi = {10.1007/s00440-021-01087-9},
	id = {Lammers2021},
	isbn = {1432-2064},
	journal = {Probability Theory and Related Fields},
	title = {Height function delocalisation on cubic planar graphs},
	ty = {JOUR},
	url = {https://doi.org/10.1007/s00440-021-01087-9},
	year = {2021}}

@article{She,
	author = {Sheffield, S.},
	journal = {Asterisque},
	number = {304},
	title = {Random surfaces},
	year = {2005}}

@article{Lis20,
	author = {Lis, Marcin},
	da = {2021/04/01},
	doi = {10.1007/s00220-021-03949-8},
	id = {Lis2021},
	isbn = {1432-0916},
	journal = {Communications in Mathematical Physics},
	number = {2},
	pages = {1181--1205},
	title = {On Delocalization in the Six-Vertex Model},
	ty = {JOUR},
	url = {https://doi.org/10.1007/s00220-021-03949-8},
	volume = {383},
	year = {2021}}

@unpublished{GlaPel,
	author = {Glazman, A. and Peled, R.},
	note = {arXiv:1909.03436},
	title = {{On the transition between the disordered and antiferroelectric phases of the 6-vertex model}},
	year = {2018}}

@article{FK,
	author = {Fortuin, C. M. and Kasteleyn, P. W.},
	journal = {Physica},
	number = {4},
	pages = {536--564},
	publisher = {Elsevier},
	title = {{On the random-cluster model: I. Introduction and relation to other models}},
	volume = {57},
	year = {1972}}

@article{ADCS,
	author = {Aizenman, Michael and Duminil-Copin, Hugo and Sidoravicius, Vladas},
	doi = {10.1007/s00220-014-2093-y},
	issn = {1432-0916},
	journal = {Communications in Mathematical Physics},
	number = {2},
	pages = {719--742},
	title = {{Random Currents and Continuity of Ising Model's Spontaneous Magnetization}},
	url = {http://dx.doi.org/10.1007/s00220-014-2093-y},
	volume = {334},
	year = {2015}}

@article{Aiz82,
	author = {Aizenman, M.},
	coden = {CMPHAY},
	fjournal = {Communications in Mathematical Physics},
	issn = {0010-3616},
	journal = {Comm. Math. Phys.},
	mrclass = {81E25 (82A68)},
	mrnumber = {678000 (84f:81078)},
	mrreviewer = {C. A. Hurst},
	number = {1},
	pages = {1--48},
	title = {Geometric analysis of {$\varphi ^{4}$} fields and {I}sing models. {I}, {II}},
	url = {http://projecteuclid.org/getRecord?id=euclid.cmp/1103921614},
	volume = {86},
	year = {1982}}

@article{Pei36,
	author = {Peierls, R.},
	journal = {Math. Proc. Camb. Phil. Soc.},
	pages = {477--481},
	title = {On {I}sing's model of ferromagnetism.},
	volume = {32},
	year = {1936}}

@unpublished{Rao17,
	author = {Raoufi, A.},
	note = {arXiv:1710.07608},
	title = {{Translation Invariant Ising Gibbs States, General Setting}}}

@article{Yan52,
	author = {Yang, C.N.},
	journal = {Phys. Rev. (2)},
	pages = {808--816},
	title = {The spontaneous magnetization of a two-dimensional {I}sing model},
	volume = {85},
	year = {1952}}

@article{newman,
	author = {Newman, C. M.},
	journal = {Zeitschrift f{\"u}r Wahrscheinlichkeitstheorie und Verwandte Gebiete},
	number = {2},
	pages = {75--93},
	publisher = {Springer},
	title = {Gaussian correlation inequalities for ferromagnets},
	volume = {33},
	year = {1975}}

@article{Lieb,
	author = {Lieb, Elliott H.},
	fjournal = {Communications in Mathematical Physics},
	journal = {Comm. Math. Phys.},
	number = {2},
	pages = {127--135},
	publisher = {Springer},
	title = {{A refinement of Simon's correlation inequality}},
	url = {https://projecteuclid.org:443/euclid.cmp/1103908400},
	volume = {77},
	year = {1980}}

@article{PfiVel,
	author = {Pfister, C. -E. and Velenik, Y.},
	day = {01},
	doi = {10.1007/BF02732435},
	issn = {1572-9613},
	journal = {Journal of Statistical Physics},
	month = {Sep},
	number = {5},
	pages = {1295--1331},
	title = {{Random-cluster representation of the Ashkin-Teller model}},
	url = {https://doi.org/10.1007/BF02732435},
	volume = {88},
	year = {1997}}

@article{ADTW,
	author = {Aizenman, Michael and Duminil-Copin, Hugo and Tassion, Vincent and Warzel, Simone},
	da = {2019/06/01},
	doi = {10.1007/s00222-018-00851-4},
	id = {Aizenman2019},
	isbn = {1432-1297},
	journal = {Inventiones mathematicae},
	number = {3},
	pages = {661--743},
	title = {{Emergent planarity in two-dimensional Ising models with finite-range Interactions}},
	ty = {JOUR},
	url = {https://doi.org/10.1007/s00222-018-00851-4},
	volume = {216},
	year = {2019}}

@article{AF,
	author = {Aizenman, M. and Fern{\'a}ndez, R.},
	doi = {10.1007/BF01011304},
	issn = {1572-9613},
	journal = {Journal of Statistical Physics},
	number = {3},
	pages = {393--454},
	title = {{On the critical behavior of the magnetization in high-dimensional Ising models}},
	url = {http://dx.doi.org/10.1007/BF01011304},
	volume = {44},
	year = {1986}}

@article{Ising,
	author = {Ising, E.},
	doi = {10.1007/BF02980577},
	fjournal = {Zeitschrift f\"ur Physik},
	issn = {0044-3328},
	journal = {Z. Physik},
	month = {FEB--APR},
	pages = {253--258},
	title = {Beitrag zur {T}heorie des {F}erromagnetismus},
	volume = {31},
	year = {1925}}

@article{GriffithsCor,
	author = {Robert B. Griffiths},
	doi = {10.1063/1.1705219},
	journal = {Journal of Mathematical Physics},
	number = {3},
	pages = {478-483},
	title = {{Correlations in Ising Ferromagnets. I}},
	volume = {8},
	year = {1967}}

@article{GHS,
	author = {Griffiths, R. B. and Hurst, C. A. and Sherman, S.},
	doi = {http://dx.doi.org/10.1063/1.1665211},
	journal = {Journal of Mathematical Physics},
	number = {3},
	pages = {790-795},
	title = {{Concavity of Magnetization of an Ising Ferromagnet in a Positive External Field}},
	url = {http://scitation.aip.org/content/aip/journal/jmp/11/3/10.1063/1.1665211},
	volume = {11},
	year = {1970}}

@article{FisherDim,
	author = {Fisher, M. E.},
	doi = {10.1063/1.1704825},
	journal = {Journal of Mathematical Physics},
	number = {10},
	pages = {1776-1781},
	title = {{On the Dimer Solution of Planar Ising Models}},
	volume = {7},
	year = {1966}}

@book{Grimmett,
	author = {Grimmett, G.},
	publisher = {Springer},
	series = {{Grundlehren der mathematischen Wissenschaften}},
	title = {{The Random-Cluster Model}},
	volume = {333},
	year = {2006}}

@article{FKG,
	ajournal = {Comm. Math. Phys.},
	author = {Fortuin, C. M. and Kasteleyn, P. W. and Ginibre, J.},
	journal = {Comm. Math. Phys.},
	number = {2},
	pages = {89--103},
	publisher = {Springer},
	title = {Correlation inequalities on some partially ordered sets},
	url = {http://projecteuclid.org/euclid.cmp/1103857443},
	volume = {22},
	year = {1971}}

@article{BFS,
	author = {Brydges, D. and Fr{\"o}hlich, J. and Spencer, T.},
	coden = {CMPHAY},
	fjournal = {Communications in Mathematical Physics},
	issn = {0010-3616},
	journal = {Comm. Math. Phys.},
	mrclass = {82A67 (81E10)},
	mrnumber = {648362 (83i:82032)},
	mrreviewer = {Gerhard C. Hegerfeldt},
	number = {1},
	pages = {123--150},
	title = {The random walk representation of classical spin systems and correlation inequalities},
	url = {http://projecteuclid.org/getRecord?id=euclid.cmp/1103920749},
	volume = {83},
	year = {1982}}

@article{GBK,
	author = {Boel, R.J. and Groeneveld, J. and Kasteleyn, P.W.},
	journal = {Physica 93A},
	pages = {138--154},
	publisher = {North-Holland Publiszhing Co.},
	title = {{Correlation-function Identities for General Planar Ising Systems}},
	year = {1978}}

@article{ABF,
	author = {{Aizenman, M. and Barsky, D.J. and Fern{\'a}ndez, R.}},
	journal = {{Journal of Statistical Physics}},
	number = {{3--4}},
	pages = {{343--374}},
	title = {{The phase transition in a general class of Ising-type models is sharp}},
	volume = {{47}},
	year = {{1987}}}

@article{KacWard,
	author = {Kac, M. and Ward, J. C.},
	doi = {10.1103/PhysRev.88.1332},
	fjournal = {Physical Review},
	issn = {0031-899X},
	journal = {Phys. Rev.},
	number = {6},
	pages = {1332--1337},
	title = {A Combinatorial Solution of the Two-Dimensional {I}sing Model},
	unique-id = {ISI:A1952UB44700023},
	volume = {88},
	year = {1952}}

@article{gunaratnam2024existence,
	author = {Gunaratnam, Trishen S and Krachun, Dmitrii and Panagiotis, Christoforos},
	journal = {Probability and Mathematical Physics},
	number = {3},
	pages = {785--845},
	publisher = {Mathematical Sciences Publishers},
	title = {{Existence of a tricritical point for the Blume--Capel model on ${\mathbb Z}^d$}},
	volume = {5},
	year = {2024}}

@article{simon1973varphi4,
	author = {Simon, Barry and Griffiths, Robert B},
	journal = {Communications in Mathematical Physics},
	number = {2},
	pages = {145--164},
	publisher = {Springer},
	title = {{The $(\varphi^4)_2$ field theory as a classical Ising model}},
	volume = {33},
	year = {1973}}

@article{newman1976rigorous,
	author = {Newman, Charles M},
	journal = {Journal of Statistical Physics},
	number = {5},
	pages = {399--406},
	publisher = {Springer},
	title = {{Rigorous results for general Ising ferromagnets}},
	volume = {15},
	year = {1976}}

@inproceedings{duminil2018random,
	author = {Duminil-Copin, Hugo},
	booktitle = {European congress of mathematics},
	organization = {European Mathematical Society-EMS-Publishing House GmbH},
	pages = {869--889},
	title = {{Random currents expansion of the Ising model}},
	year = {2018}}

@article{duminil2019sharp,
	author = {Duminil-Copin, Hugo and Raoufi, Aran and Tassion, Vincent},
	journal = {Annals of Mathematics},
	number = {1},
	pages = {75--99},
	publisher = {JSTOR},
	title = {{Sharp phase transition for the random-cluster and Potts models via decision trees}},
	volume = {189},
	year = {2019}}

@article{aoun2024phase,
	author = {Aoun, Yacine and Dober, Moritz and Glazman, Alexander},
	journal = {Communications in Mathematical Physics},
	number = {2},
	pages = {37},
	publisher = {Springer},
	title = {{Phase diagram of the Ashkin--Teller model}},
	volume = {405},
	year = {2024}}

@article{duminil2020exponential,
	author = {Duminil-Copin, Hugo and Goswami, Subhajit and Raoufi, Aran},
	journal = {Communications in mathematical physics},
	number = {2},
	pages = {891--921},
	publisher = {Springer},
	title = {{Exponential decay of truncated correlations for the Ising model in any dimension for all but the critical temperature}},
	volume = {374},
	year = {2020}}

@article{simon1980correlation,
	author = {Simon, Barry},
	journal = {Communications in Mathematical Physics},
	number = {2},
	pages = {111--126},
	publisher = {Springer},
	title = {Correlation inequalities and the decay of correlations in ferromagnets},
	volume = {77},
	year = {1980}}

@book{simon2026phase,
	author = {Simon, Barry},
	publisher = {Cambridge University Press},
	title = {Phase transitions in the theory of lattice gases},
	year = {2026}}

@article{ellis1976ghs,
	author = {Ellis, Richard S and Monroe, James L and Newman, Charles M},
	journal = {Communications in Mathematical Physics},
	number = {2},
	pages = {167--182},
	publisher = {Springer},
	title = {{The GHS and other correlation inequalities for a class of even ferromagnets}},
	volume = {46},
	year = {1976}}

@article{panagiotis2026subcritical,
	author = {Panagiotis, Christoforos and Veitch, William},
	journal = {arXiv preprint arXiv:2605.31344},
	title = {Subcritical sharpness for real-valued spin models},
	year = {2026}}

@article{glazman2025delocalisation,
	author = {Glazman, Alexander and Lammers, Piet},
	journal = {Communications in Mathematical Physics},
	number = {5},
	pages = {108},
	publisher = {Springer},
	title = {{Delocalisation and continuity in 2D: loop $O(2)$, six-vertex, and random-cluster models}},
	volume = {406},
	year = {2025}}

@inproceedings{harris1960lower,
	author = {Harris, Theodore E},
	booktitle = {Mathematical Proceedings of the Cambridge Philosophical Society},
	number = {1},
	organization = {Cambridge University Press},
	pages = {13--20},
	title = {A lower bound for the critical probability in a certain percolation process},
	volume = {56},
	year = {1960}}

@article{brydges1983random,
	author = {Brydges, David C and Fr{\"o}hlich, J{\"u}rg and Sokal, Alan D},
	journal = {Communications in mathematical physics},
	number = {1},
	pages = {117--139},
	publisher = {Springer},
	title = {{The random-walk representation of classical spin systems and correlation inequalities: II. The skeleton inequalities}},
	volume = {91},
	year = {1983}}

@article{lammers2022dichotomy,
	author = {Lammers, Piet},
	journal = {arXiv preprint arXiv:2211.14365},
	title = {A dichotomy theory for height functions},
	year = {2022}}

@article{glazman2025planar,
	author = {Glazman, Alexander and Harel, Matan and Zelesko, Nathan},
	journal = {arXiv preprint arXiv:2508.20917},
	title = {{Planar percolation and the loop $O(n)$ model}},
	year = {2025}}

@article{aizenman2020exponential,
	author = {Aizenman, Michael and Harel, Matan and Peled, Ron},
	journal = {Journal of Statistical Physics},
	number = {1},
	pages = {304--331},
	publisher = {Springer},
	title = {{Exponential decay of correlations in the 2d random field Ising model}},
	volume = {180},
	year = {2020}}

@article{panagiotis2026regularity,
	author = {Panagiotis, Christoforos and Veitch, William},
	journal = {arXiv preprint arXiv:2603.26319},
	title = {{Regularity of Gibbs measures for unbounded spin systems on general graphs}},
	year = {2026}}

@article{lis2019circle,
	author = {Lis, Marcin},
	journal = {Communications in Mathematical Physics},
	number = {2},
	pages = {507--530},
	publisher = {Springer},
	title = {{Circle Patterns and Critical Ising Models}},
	volume = {370},
	year = {2019}}

@article{lebowitz1976statistical,
	author = {Lebowitz, Joel L and Presutti, Errico},
	journal = {Communications in Mathematical Physics},
	number = {3},
	pages = {195--218},
	publisher = {Springer},
	title = {Statistical mechanics of systems of unbounded spins},
	volume = {50},
	year = {1976}}

@misc{gunaratnam2026supercriticalsharpnessrandomcluster,
	archiveprefix = {arXiv},
	author = {Trishen S. Gunaratnam and Dmitry Krachun and Christoforos Panagiotis and Romain Panis and Franco Severo},
	eprint = {2608.18045},
	primaryclass = {math.PR},
	title = {Supercritical sharpness for the random cluster representation of real-valued spin models},
	url = {https://arxiv.org/abs/2608.18045},
	year = {2026}}
\end{document}